\documentclass[a4paper,oneside]{amsart}

\usepackage[T1]{fontenc}
\usepackage{lmodern}
\usepackage{fixcmex}
\usepackage{microtype}

\usepackage{amsmath,amsthm,amssymb,amsfonts,amscd}
\usepackage{mathtools}
\usepackage{xfrac}
\usepackage{xcolor}
\usepackage{bm}
\usepackage[maxbibnames=99,maxalphanames=99,giveninits=true,style=alphabetic,citestyle=alphabetic,doi=false,isbn=false,url=false]{biblatex}
\AtEveryBibitem{\clearlist{language}}
\usepackage[hypertexnames=false]{hyperref}

\usepackage[shortlabels]{enumitem}

\usepackage{epsfig}
\usepackage{graphicx}
\usepackage{mathrsfs}
\usepackage{pinlabel}
\usepackage[outercaption]{sidecap}
\usepackage{subcaption}

\definecolor{light-gray}{gray}{0.60}
\definecolor{heavyred}{rgb}{0.75,0,0}
\definecolor{heavygreen}{rgb}{0,0.75,0}
\definecolor{heavyblue}{rgb}{0,0,0.75}

\newcounter{notes}

\theoremstyle{plain}
\newtheorem{theorem}{Theorem}[section]
\newtheorem{lemma}[theorem]{Lemma}
\newtheorem{proposition}[theorem]{Proposition}
\newtheorem{corollary}[theorem]{Corollary}
\newtheorem{fact}[theorem]{Fact}

\theoremstyle{definition}
\newtheorem{definition}[theorem]{Definition}
\newtheorem{example}[theorem]{Example}
\newtheorem{remark}[theorem]{Remark}

\numberwithin{equation}{section}

\renewcommand{\leq}{\leqslant}
\renewcommand{\geq}{\geqslant}
\renewcommand{\epsilon}{\varepsilon}
\renewcommand{\phi}{\varphi}

\newcommand{\C}{\mathbb{C}}
\newcommand{\R}{\mathbb{R}}
\newcommand{\N}{\mathbb{N}}
\newcommand{\Z}{\mathbb{Z}}

\DeclareMathOperator{\SL}{SL}
\DeclareMathOperator{\PSL}{PSL}
\DeclareMathOperator{\GL}{GL}
\DeclareMathOperator{\PGL}{PGL}
\DeclareMathOperator{\PO}{PO}
\let\O\relax
\DeclareMathOperator{\O}{O}
\DeclareMathOperator{\Isom}{Isom}
\DeclareMathOperator{\Hom}{Hom}
\DeclareMathOperator{\Proj}{\mathbb{P}}

\newcommand{\RP}{\mathbb{RP}}
\renewcommand{\H}[1][d]{\mathbb{H}^{#1}}
\newcommand{\W}{\mathcal{W}}
\newcommand{\HLP}{\mathcal{HLP}}
\newcommand{\QP}{\mathbb{P}^1(\mathbb{Q})}

\newcommand{\inn}{\mathrm{inn}}
\DeclareMathOperator{\PrimSys}{PSys}
\DeclareMathOperator{\Christoffel}{Ch}
\DeclareMathOperator{\Ax}{Axis}
\DeclareMathOperator{\Lv}{Lv}

\DeclareMathOperator{\tr}{tr}
\DeclareMathOperator{\Perp}{\perp}
\newcommand{\Lc}{\mathcal{L}}
\newcommand{\hem}[2]{{\mathcal H}_{#1}^{#2}}
\newcommand{\ohem}[2]{\mathrm{Int}(\hem{#1}{#2})}
\newcommand{\hex}{\mathcal{H}}
\newcommand{\hyp}[2]{\mathcal{P}_{#1}^{#2}}
\renewcommand{\P}{\mathcal{P}}
\newcommand{\Q}{\mathcal{Q}}
\newcommand{\Cc}{\mathcal{C}}
\newcommand{\geom}[2]{\ell_{#1}^{\textrm{geom}}(#2)}
\newcommand{\trans}[2][\H]{\ell_{#1}(#2)}
\newcommand{\transhex}[1]{\ell(#1)}
\newcommand{\hor}{\mathcal{H}}
\DeclareMathOperator{\Span}{Span}
\newcommand{\side}[2][]{H_{#2}^{#1}}
\DeclareMathOperator{\Out}{Out}
\DeclareMathOperator{\Aut}{Aut}
\DeclareMathOperator{\Inn}{Inn}
\newcommand{\id}{\mathrm{id}}

\newcommand{\tosub}[1]{\xrightarrow[#1]{}}

\DeclarePairedDelimiter\abs{\lvert}{\rvert}
\DeclarePairedDelimiter\norm{\lVert}{\lVert}
\DeclarePairedDelimiter\gen{\langle}{\rangle}

\newcommand{\resp}{resp.~}
\newcommand{\ie}{\emph{i.e.}~}
\newcommand{\Ch}{Ch.~}
\newcommand{\Lm}{Lm.~}
\newcommand{\Thm}{Thm.~}
\newcommand{\Cor}{Cor.~}
\newcommand{\Rmk}{Rmk.~}
\newcommand{\Prop}{Prop.~}
\newcommand{\Exe}{Exercise~}
\newcommand{\pp}{p.~}

\title[Primitive stability and the $Q$-conditions in $\H$]{Primitive stability and the $Q$-conditions for the rank two free group in hyperbolic $d$-space}
\date{\today}

\author[B. Fléchelles]{Balthazar Fléchelles}
\address{Institut Fourier, Université Grenoble Alpes, CS 40700, 38058 Grenoble Cedex 09}
\email{balthazar.flechelles@univ-grenoble-alpes.fr}

\thanks{The author was supported by the ANR-23-CE40-0012 HilbertXfield and the grant NRF-2022R1I1A1A01072169, and received funding from the European Research Council (ERC) under the European Union's Horizon 2020 research and innovation programme (ERC starting grant DiGGeS, grant agreement No 715982, and ERC consolidator grant GeometricStructures, grant agreement No 614733), as well as the Programme Visibilité Scientifique Junior of the Labex Mathématiques Hadamard (2022).} 


\begin{document}

\begin{abstract}
   The two largest known domains of discontinuity for the action of $\Out(F_2)$ on the $\PSL_2(\C)$-character variety of $F_2$ --- defined by Minsky's primitive stability, and Bowditch's $Q$-conditions --- were proven to be equal independently by Lee--Xu and Series. 
   We prove the equivalence between primitive stability and a generalization of the $Q$-conditions for representations of $F_2$ into the isometry group of hyperbolic $d$-space for $d\geq 3$, under some assumptions. In particular, these assumptions are satisfied by all $W_3$-extensible representations.

   We also generalize Lee--Xu's and Series' results concerning the bounded intersection property to higher dimensions after extending their original definition to this setting.
\end{abstract}


\maketitle

\tableofcontents

\section{Introduction}

Let $F_n$ denote the free group on $n$ generators, and $\H$ the real hyperbolic space of dimension $d$. We aim to study the dynamics of the action of the outer automorphism group $\Out(F_n)$ on the character variety $\chi(F_n,\Isom(\H))$. The character variety is formed as the largest Hausdorff quotient of the set $\Hom(F_n,\Isom(\H))/\Isom(\H)$ of representations from $F_n$ into the isometry group $\Isom(\H)$ of $\H$, modulo conjugation at the target. The action of $\Aut(F_n)$ on $\Hom(F_n,\Isom(\H))$ by precomposition yields an action of $\Out(F_n):=\Aut(F_n)/\Inn(F_n)$ on $\chi(F_n,\Isom(\H))$.

The question of finding the largest (if it exists) \emph{domain of discontinuity} for this action, \ie an open invariant subset of $\chi(F_n,\Isom(\H))$ on which the action is properly discontinuous, has proven to be of interest. For a long time, the set of convex cocompact representations was the largest known domain of discontinuity, following works of Fricke and Klein \cite{FrickeKlein1897} for $d=2$ and Bers and Sullivan \cite{Bers1970,Sullivan1981} for $d=3$.

More recently, Tan, Wong and Zhang \cite{TanWongZhang2008GeneralizedMarkoffMaps} and Minsky \cite{Minsky2013} defined domains of discontinuity that are strictly larger than the convex cocompact locus, in various contexts.
They involve the notion of a \emph{primitive element} of $F_n$, \ie a member of a \emph{basis} of $F_n$, that is, a generating set of cardinal $n$.

\subsection{The \texorpdfstring{$Q$}{Q}-conditions}

Tan, Wong and Zhang's domain of discontinuity is a subset of $\chi(F_2,\allowbreak\PSL_2(\C))$, defined by the \emph{$Q$-conditions}. They generalize earlier work of Bowditch who focused on the case of the type preserving representations of the fundamental group of the one-holed torus (see \cite{TanWongZhang2008GeneralizedMarkoffMaps,Bowditch1998}).

\begin{definition}\label{def:QCondH3orig}
   A representation $\rho: F_2\to\PSL_2(\C)$ satisfies the \emph{$Q$-conditions} if
   \begin{enumerate}
      \item $\rho(w)$ is hyperbolic for all primitive elements $w\in F_2$;
      \item the set of conjugacy classes of primitive elements $w\in F_2$ such that $\abs{\tr\rho(w)} \leq 2$ is finite.
   \end{enumerate}
\end{definition}

Observe that the $Q$-conditions can be seen as a weakening of a characterization of convex cocompactness: if we replace the words “primitive elements” by “nontrivial elements” in Definition \ref{def:QCondH3orig}, then work of Goldman \cite{Goldman2003} shows that a representation $\rho: F_2\to\PSL_2(\R)$ is convex cocompact if and only if it satisfies these modified conditions.

Clearly any convex cocompact representation of $F_2$ into $\PSL_2(\C)$ satisfies the $Q$-conditions. On the other hand, the aforementioned work of Goldman \cite{Goldman2003} shows that many non-discrete representations also satisfy the $Q$-conditions. Hence the domain of discontinuity defined by the $Q$-conditions is strictly larger than the convex cocompact locus.

While this definition apparently relies on the target space being $\PSL_2(\C)$, one can try to generalize it to more general representations of $F_2$ by remembering that the trace of a hyperbolic element $A$ of $\PSL_2(\C)$ is closely linked to its \emph{translation length} $\trans[{\H[3]}]{A}:=\inf\{d(o,Ao),o\in\H[3]\}$ in $\H[3]$. Recall that in a proper geodesic Gromov hyperbolic space $X$, an isometry $\gamma\in\Isom(X)$ is \emph{hyperbolic} if and only if its orbits are quasi-geodesics (see \cite[\S 21 \pp 150]{GhysDelaharpe1990}).

\begin{definition}\label{def:BQcond}
   Let $X$ be a proper geodesic Gromov hyperbolic space and $\lambda>0$. A representation $\rho: F_n\to\Isom(X)$ satisfies the \emph{$Q_\lambda$-conditions} if
   \begin{enumerate}
      \item $\rho(w)$ is hyperbolic for all primitive elements $w\in F_n$;
      \item the set of conjugacy classes of primitive elements with translation length $\trans[X]{\rho(w)}\leq\lambda$ is finite.
   \end{enumerate}
\end{definition}

Notice that if we choose $\lambda$ large enough, then the $Q_\lambda$-conditions imply the $Q$-conditions for characters in $\chi(F_2,\PSL_2(\C))$. It turns out the reciprocal also holds by Tan--Wong--Zhang's work on primitive displacement.

\subsection{Primitive displacement}

Tan, Wong and Zhang \cite{TanWongZhang2008GeneralizedMarkoffMaps} showed that the $Q$-conditions are equivalent to \emph{primitive displacement} in $\chi(F_2,\PSL_2(\C))$.

\begin{definition}\label{def:primDisp}
   Let $X$ be a proper geodesic Gromov hyperbolic space. A representation $\rho: F_n\to\Isom(X)$ is \emph{primitive displacing} if there are constants $K,c>0$ such that for all primitive element $w\in F_n$, $\rho(w)$ is hyperbolic, and
   \begin{equation*}
      \trans[X]{\rho(w)}\geq K\norm{w} - c
   \end{equation*}
   where $\norm{w}:=\inf\{\abs{uwu^{-1}},u\in F_n\}$ denotes the reduced word length of $w$.
\end{definition}

Once more, replacing the condition that $w$ is primitive by the condition that $w$ is nontrivial gives the characterization of convex cocompact representations as well-displacing representations given in \cite{DelzantGuichardLabourieMozes2011}. 
We observe that primitive displacement implies the $Q_\lambda$-conditions for all positive $\lambda$, as there are finitely many conjugacy classes of primitive elements of $F_n$ whose reduced length is smaller than any given constant. 

\subsection{Primitive stability}

Although Minsky's domain of discontinuity was introduced later, it is defined in a broader setting than Tan, Wong and Zhang's. The associated condition on representations is \emph{primitive stability}. Recall that any $w\in F_n$ preserves a unique biinfinite geodesic called its axis in the Cayley graph of $F_n$ with the standard set of generators.

\begin{definition}\label{def:primStab}
   Let $X$ be a proper geodesic Gromov hyperbolic space. A representation $\rho:F_n\to \Isom(X)$ is \emph{primitive stable} if there is a point $o\in X$ and constants $K,c>0$ such that, for every primitive $w\in F_2$, the orbit map based at $o$ sends the axis of $w$ to a $(K,c)$-quasi-geodesic.
\end{definition}

Note that replacing the condition “primitive” with “nontrivial” gives the classical characterization of convex cocompactness in terms of quasi-isometric embeddings. Hence the primitive stable locus clearly contains the convex cocompact locus. In fact, Minsky showed that for all $d\geq 3$, the primitive stable locus in $\chi(F_n,\allowbreak\Isom(\H))$ is strictly larger than the set of convex cocompact representations.

This always defines a domain of discontinuity in $\chi(F_n,\Isom(X))$, as proven by Minsky \cite{Minsky2013} for rank $1$ symmetric spaces, and in the general Gromov hyperbolic case by Remfort-Aurat in \cite{RemfortAurat2023}. 

\subsection{Known implications}

The fact that the $Q$-conditions, primitive stability and primitive displacement can all be interpreted as weakenings of convex cocompactness by “restricting to primitive elements” seems to indicate that these conditions might be related. In fact, it is elementary to observe that primitive stability implies primitive displacement, which implies the $Q_\lambda$-conditions for all $\lambda>0$ (see Proposition \ref{prop:PSinPDinBQ}). It is natural to ask whether the $Q_\lambda$-conditions and primitive stability are equivalent.

In his thesis, Lupi \cite{Lupi2015} showed that the $Q$-conditions are equivalent to primitive stability in $\chi(F_2,\PSL_2(\R))$. Lee and Xu \cite{LeeXu2020}, Series \cite{Series2019,Series2020} and Schlich \cite{Schlich2022} all proved that primitive displacement and primitive stability are equivalent, in $\chi(F_2,\PSL_2(\C))$ for Lee--Xu and Series, and in the general Gromov hyperbolic case for Schlich.

Observe that in $\chi(F_2,\PSL_2(\C))$, the result of Lee, Xu and Series, together with the proof by Tan, Wong and Zhang of the equivalence between primitive displacement and the $Q$-conditions yields the equivalence between the $Q$-conditions and primitive stability. In particular, if $\lambda$ is large enough, then the $Q_\lambda$-conditions are equivalent to primitive stability in $\chi(F_2,\PSL_2(\C))$.

\subsection{Results}

This leaves open the more general question of whether primitive stability and the $Q_\lambda$-conditions are equivalent in $\chi(F_n,\Isom(X))$, where $n\geq 2$ and $X$ is a proper geodesic Gromov hyperbolic space. In this paper, we focus on the special case where $n=2$ and $X=\H$ for $d\geq 3$. In the proofs of Lee--Xu and Series, they implicitly use the hypothesis that their representations are \emph{Coxeter extensible}. It is indeed a classical fact that all actions of $F_2$ on $\H[3]$ are Coxeter extensible (see \cite{Fenchel1989,Goldman2009}), though this fails in higher dimensions (see Remark \ref{rmk:CoxeterExtensibleRepsDimEstimate}). Recall from Mühlherr \cite{Muehlherr1997} that if $W_3 = \gen{x,y,z\mid x^2,y^2,z^2}$ is the universal Coxeter group of rank $3$, then the morphism $\tau: F_2=\gen{a,b}\to W_3$ sending $a$ to $xy$ and $b$ to $yz$ is an embedding.

\begin{definition}\label{def:CoxExtRep}
   A representation $\rho:F_2\to\Isom(\H)$ is \emph{Coxeter extensible} if there is a representation $\rho':W_3\to\Isom(\H)$ such that $\rho = \rho'\circ\tau$.
\end{definition}

\begin{theorem}[see Theorem \ref{thm:PSisPDisBQ}]\label{thm:PSisBQCox}
   Let $\rho: F_2\to\Isom(\H)$ be a Coxeter extensible representation. There exists $\lambda >0$ such that $\rho$ is primitive stable if and only if it satisfies the $Q_\lambda$-conditions.
\end{theorem}

\begin{remark}
   Theorem \ref{thm:PSisPDisBQ} is stated in a more general setting by replacing the assumption that $\rho$ is Coxeter extensible by a weaker condition which we call the half-length property (see Section \ref{sec:HLP} for the definition).
\end{remark}

\begin{remark}\label{rmk:precisionMainThmCox}
   In Theorem \ref{thm:PSisBQCox}, $\lambda$ depends on the choice of $\rho$. More precisely, we show in Theorem \ref{thm:PSisPDisBQ} that we can choose $\lambda = \Lc_d([\rho])$, where $\Lc_d:\chi(F_2,\Isom(\H))\to \R_{>0}$ is an $\Out(F_2)$-invariant function that depends only on the \emph{primitive systole} of $\rho$
   \begin{equation*}
      \PrimSys([\rho]):=\inf\{\trans{\rho(x)}, \text{$x\in F_2$ primitive}\}
   \end{equation*}
   Since this quantity is invariant under the action of $\Aut(F_2)$, the set of characters $[\rho]$ of $\chi(F_2,\Isom(\H))$ satisfying the $Q_{ \Lc_d([\rho])}$-conditions indeed defines a $\Out(F_2)$-invariant subset of $\chi(F_2,\Isom(\H))$.
    In fact, $\Lc_d([\rho])\to + \infty$ as $\PrimSys([\rho])\to 0$, hence the smaller the primitive systole of $\rho$, the larger $\lambda$ needs to be in Theorem \ref{thm:PSisBQCox}.
\end{remark}

\begin{remark}
   Instead of considering Coxeter extensible representations of $F_2$, we may as well consider representations of $W_3$. We say that $\rho: W_3\to\Isom(\H)$ is primitive stable or satisfies the $Q_\lambda$-conditions if $\rho\circ\tau$ does. As a consequence of classical results by Mühlherr \cite{Muehlherr1997}, one can check that the primitive stable locus is a domain of discontinuity for the action of $\Out(W_3)$ on $\chi(W_3,\Isom(\H))$ which is strictly larger than the convex cocompact locus (see Proposition \ref{prop:linkW3F2CharVar}). In this setting, Theorem \ref{thm:PSisBQCox} says that the image in $\chi(W_3,\Isom(\H))$ of the set of representations $\rho$ satisfying the $Q_{ \Lc_d([\rho\circ\tau])}$-conditions coincides with the primitive stable locus.
\end{remark}

\begin{remark}
   Theorem \ref{thm:PSisBQCox} recovers Schlich's result when $X=\H$ and $d\geq 3$, as well as Lee, Xu and Series' result and part of Tan, Wong and Zhang's proof of the equivalence between the $Q$-conditions and primitive displacement in $\chi(F_2,\PSL_2(\C))$.
\end{remark}

Interestingly, Lee, Xu and Series all introduced a new $\Out(F_2)$-invariant subset of $\chi(F_2,\PSL_2(\C))$ which contains the primitive stable locus. This subset, defined by the \emph{bounded intersection property} (see Definition \ref{def:BIP}), can also be seen as a weakening by “restricting to the primitive elements” of a characterization of convex cocompactness in $\chi(F_2,\PSL_2(\C))$ due to Gilman and Keen \cite{GilmanKeen2009}. It is unknown whether this defines a domain of discontinuity, or even if it is strictly larger than the primitive stable locus. We are able to generalize their definition and results to $\chi(F_2,\Isom(\H))$ for $d\geq 3$.

\begin{theorem}\label{thm:PSinBIP}
   Primitive stable representations $F_2\to\Isom(\H)$ satisfy the bounded intersection property. Moreover, discrete and faithful representations that have the bounded intersection property satisfy the $Q_\lambda$-conditions for all $\lambda>0$.
\end{theorem}
\begin{remark}
   Series \cite{Series2019,Series2020} proves that the $Q$-conditions imply the bounded intersection property. Her proof relies both on the existence of right-angled hexagons and associated trigonometric formulae and on the equivalence of the $Q$-conditions with primitive displacement in $\chi(F_2,\PSL_2(\C))$ proven in \cite{TanWongZhang2008GeneralizedMarkoffMaps}. Hence, it may be possible to generalize her proof to get that primitive displacement implies the bounded intersection property for Coxeter extensible representations, or perhaps even for representations satisfying the half-length property.
\end{remark}

The proof of Theorem \ref{thm:PSisBQCox} is an adaptation of the proof of Lee and Xu. The main ingredient of their proof and ours is the geometry of right-angled hexagons in $\H$, which are defined for Coxeter extensible representations. This is why representations are assumed to be Coxeter extensible in Theorem \ref{thm:PSisBQCox}. We are able to weaken that condition and work with hexagons that are much less regular a priori, allowing to ask for the weaker half-length property instead.

One apparent obstacle is that they make heavy use of trigonometric formulae for right-angled hexagons in $\H[3]$ in a crucial lemma. Such formulae are classical for $d\leq 3$, and Tan, Wong and Zhang \cite{TanWongZhang2012} extended them to $d=4$, but nothing is known in higher dimensions. We are able to circumvent this issue by considering instead the large scale geometry of our hexagons.

The paper is organized as follows. In Section \ref{sec:primElem}, we briefly recall the useful facts about primitive elements of $F_2$ and make the link between the action of the outer automorphism groups of $F_2$ and $W_3$ on their respective character varieties. In Section \ref{sec:HLP}, we introduce the half-length property and show that all Coxeter extensible representations satisfy it. In Section \ref{sec:PSinBQ}, we prove and recall the elementary implications in the proof of Theorem \ref{thm:PSisBQCox}, and present our main result: Theorem \ref{thm:PSisPDisBQ}. In Section \ref{sec:proofOfConverse}, we prove the last implication of our main theorem. In Section \ref{sec:BIP}, we briefly recall and discuss Lee, Xu and Series' definition of the bounded intersection property, and we prove Theorem \ref{thm:PSinBIP}. In Section \ref{sec:furtherQuestions}, we give a few directions at which further research may be aimed.

\subsection*{Acknowledgments}

The author warmly thanks Anna Wienhard for welcoming him in her team under the excellent supervision of Gye-Seon Lee, for much of this work happened during this stay. 
We are grateful to Konstantinos Tsouvalas for his help in proving a lemma, and to Jaejeong Lee, Anne Parreau and Suzanne Schlich for extensive discussions and help in proofreading. Finally, the author would like to express his deepest gratitude to Fanny Kassel and Gye-Seon Lee for taking the time to discuss the different stages of this paper, and without whom it may have never been.

\section{Automorphisms and primitive elements}\label{sec:primElem}

In this section, we study the structure of the (outer) automorphism groups of $F_2$ and $W_3$ in order to show how to build domains of discontinuity for the action of $\Out(W_3)$ on $W_3$-characters by using existing domains of discontinuity for the action of $\Out(F_2)$ on $F_2$-characters. We also derive from this study the main results about primitive elements of $F_2$ that we will need in this paper.

Much of this discussion was inspired by Section $2$ and paragraphs A.1 and A.2 of \cite{LeeXu2020}. The main results are from \cite{Nielsen1917} and \cite{Muehlherr1997}.

\subsection{Primitive elements, bases and automorphisms}

We choose and fix a preferred basis $e=(a,b)$ of $F_2$ for the rest of the paper. We recall some of the definitions given in the introduction.

\begin{definition}
An element $u\in F_2$ is said to be \emph{primitive} if there is a $v\in F_2$ such that the pair $(u,v)$ generates $F_2$.

   Such a pair is called a \emph{basis} of $F_2$.
\end{definition}
\begin{remark}
Observe that if $u\in F_2$ is primitive, then so are $u^{-1}$ and $wuw^{-1}$ for every $w\in F_2$.
\end{remark}

Primitive elements will play a central role in the rest of this paper, and a careful study of them is necessary. It is easier to deal with bases, as a single primitive element is part of infinitely many bases (even modulo simultaneous conjugation).

It is useful to observe that bases of $F_2$ are in bijection with automorphisms of $F_2$. Indeed, the image of $e=(a,b)$ under an automorphism is always a basis, and conversely, there is a unique automorphism sending $e$ to a given basis. In his seminal paper, Nielsen gives a precise description of the automorphisms of $F_2$.

\begin{fact}[\cite{Nielsen1917}]\label{fact:Nielsen}
   The group $\Aut(F_2)$ is generated by the automorphisms $S$, $I$, $L$ and $R$ of $F_2$ respectively sending $(a,b)$ to
   \begin{align*}
      &(b,a)&&(a^{-1},b)&&(a,ab)&&(ab,b)
   \end{align*}
\end{fact}

These four automorphisms are called the \emph{Nielsen transformations} of $F_2$. A pair of elements of $F_2$ is a basis if and only if it can be reached by a finite sequence of Nielsen transformations.

Recall from Definition \ref{def:CoxExtRep} the morphism $\tau: F_2\to W_3$ defined by $\tau(a) = xy$ and $\tau(b) = yz$. This is a well-behaved embedding by work of Mühlherr.

\begin{fact}[\cite{Muehlherr1997}]\label{fact:Muehlherr}
   $\tau$ is an embedding whose image is characteristic in $W_3$. The induced morphism $\iota : \phi\in\Aut(W_3)\to\phi|_{\tau(F_2)}\in\Aut(F_2)$ is an embedding.
\end{fact}

\begin{remark}
   Mühlherr's theorem is more general, as it deals with $\tau: F_n\to W_{n+1}$, proving that it is an embedding onto a characteristic subgroup which induces an embedding $\iota: \Aut(W_{n+1})\to\Aut(F_n)$ of the automorphism groups. He also gives a presentation of $\Aut(W_{n+1})$.
\end{remark}

Fact \ref{fact:Muehlherr} allows to write the short exact sequence
\begin{equation*}
   0\to F_2\overset{\tau}{\to} W_3\to \Z/2\Z\to 0
\end{equation*}
where the right morphism sends $x$, $y$ and $z$ to $1\in\Z/2\Z$. Since the morphism $\Z/2\Z\to W_3$ sending $1$ to $x$ is a right splitting of this sequence, we have
\begin{equation}\label{eq:W3SDPF2}
   W_3 = \tau(F_2)\rtimes \gen{x}
\end{equation}
This can also be seen by observing that $y = x\tau(a)$ and $z = x\tau(ab)$.

Observe that as for $F_2$, automorphisms of $W_3$ are uniquely determined by the image of the canonical basis. Hence $\Aut(W_3)$ is in bijection with the set of bases of $W_3$, \ie triplets of involutions $(x',y',z')\in W_3^3$ that generate $W_3$.

It is not difficult to see that each one of the Nielsen transformations of $F_2$ is induced by the action of an automorphism of $W_3$. As in Fact \ref{fact:Nielsen}, we represent an automorphism $\phi$ of $W_3$ by the basis $(\phi(x),\phi(y),\phi(z))$. The automorphisms given by the bases
\begin{align*}
   &(yzy,y,yxy)&&(yxy,y,z)&&(xyx,x,z)&&(x,z,zyz)
\end{align*}
respectively induce on $\tau(F_2)\simeq F_2$ the Nielsen transformations $S$, $I$, $L$ and $R$. Together with Fact \ref{fact:Nielsen}, this implies that $\iota$ is in fact an isomorphism. We collect this fact in the following proposition.

\begin{proposition}\label{prop:AutF2W3Same}
   The morphism $\iota:\Aut(W_3)\to\Aut(F_2)$ induced by the restriction to the image of $\tau$ is an isomorphism. In particular, we have, for all $\phi\in\Aut(F_2)$
   \begin{equation*}
      \tau\circ\phi = \iota^{-1}(\phi)\circ\tau
   \end{equation*}
\end{proposition}

\subsection{(Outer) automorphism groups}\label{par:autGps}

Observe that the abelianization morphism $\pi_e:F_2\to\Z^2$ defined by $\pi_e(a) = (1,0)$ and $\pi_e(b)=(0,1)$ induces a morphism
\begin{equation*}
   (\pi_e)_\ast:\begin{cases}
      \Aut(F_2)&\to\GL(2,\Z) \\
      \phi&\mapsto (\pi_e(w)\in\Z^2\mapsto\pi_e\circ\phi(w))
   \end{cases}
\end{equation*}
It is elementary to check that $(\pi_e)_\ast$ is well-defined once we observe that $\pi_e(w)$ is the sum of the exponents of the letters $a$ and $b$ when writing $w$ in the basis $e=(a,b)$. Note that an automorphism $\phi\in\Aut(F_2)$ is sent to the matrix $(\pi_e\circ\phi(a),\pi_e\circ\phi(b))$, where the two vectors $\pi_e\circ\phi(a)$ and $\pi_e\circ\phi(b)$ form the columns of the matrix.

Moreover, since $\GL_2(\Z)$ is generated by
\begin{align*}
   (\pi_e)_\ast(I) &=\begin{pmatrix}
      -1 & 0 \\ 0 & 1
   \end{pmatrix}&
   (\pi_e)_\ast(L)&=\begin{pmatrix}
      1 & 1 \\ 0 & 1
   \end{pmatrix}&
   (\pi_e)_\ast(R)&=\begin{pmatrix}
      1 & 0 \\ 1 & 1
   \end{pmatrix}
\end{align*}
$(\pi_e)_\ast$ is surjective. Nielsen also computes the kernel of $(\pi_e)_\ast$.

\begin{fact}[\cite{Nielsen1917}]\label{fact:NielsenKernel}
   The kernel of $(\pi_e)_\ast$ is $\Inn(F_2)$.
\end{fact}

It follows that $(\pi_e)_\ast$ induces an isomorphism
\begin{equation}\label{eq:OutF2Isom}
   \Out(F_2)\simeq\GL_2(\Z)
\end{equation}

We also want to compute $\Out(W_3)$, and relate it to $\Out(F_2)$. For any $u\in W_3$, let $\inn_u\in\Aut(W_3)$ denote the conjugation by $u$. Observe that \eqref{eq:W3SDPF2} translates to
\begin{equation*}
   \iota(\Inn(W_3)) = \Inn(F_2) \rtimes \gen{\iota(\inn_x)}
\end{equation*}
since $\iota\circ\inn: W_3\to\Aut(F_2)$ is an embedding by Fact \ref{fact:Muehlherr}. The fact that $\inn_x(xy) = yx = \tau(a^{-1})$ and $\inn_x(yz) = xyzx = xyzyyx=\tau(ab^{-1}a^{-1})$ implies that $\iota(\inn_x)$ is the automorphism of $\Aut(F_2)$ given by the basis $(a^{-1},ab^{-1}a^{-1})$. It follows that $(\pi_e)_\ast\circ\iota(\inn_x) = -\id$.

Therefore, if we let $p:\GL_2(\Z)\to\PGL_2(\Z)$ be the quotient map, we have $\ker(p\circ (\pi_e)_\ast\circ \iota) = \Inn(W_3)$, so that $p\circ(\pi_e)_\ast\circ\iota$ induces an isomorphism
\begin{equation}\label{eq:OutW3Isom}
   \Out(W_3)\simeq\Out(\tau(F_2))/\gen{[\inn_x]}\simeq\PGL_2(\Z)
\end{equation}
where $[\inn_x]$ denotes the class of $\inn_x|_{\tau(F_2)}\in\Aut(\tau(F_2))$ in $\Out(\tau(F_2))$.

\subsection{Conjugated actions on character varieties}

In this paragraph, we let $G$ denote any group. We use the results of the previous paragraphs to establish the link between domains of discontinuity for the action of $\Out(W_3)$ on $\chi(W_3,G)$ and of $\Out(F_2)$ on $\chi(F_2,G)$.

We have a natural continuous map $\tau^\ast:\chi(W_3,G)\to\chi(F_2,G)$ induced by the map $\rho\mapsto\rho\circ\tau$ defined for all $\rho : W_3\to G$. By Proposition \ref{prop:AutF2W3Same}, $\tau^\ast$ is $\iota^{-1}$-equivariant: for all $\phi\in\Aut(F_2)$ and $[\rho]\in\chi(W_3,G)$, $\phi\cdot\tau^\ast([\rho]) = \tau^\ast(\iota^{-1}(\phi)\cdot [\rho])$. Hence, we have for all $[\phi]\in\Out(F_2)$:
\begin{equation}
   [\phi]\cdot\tau^\ast([\rho]) = \tau^\ast(\iota^\ast([\phi])\cdot [\rho]) \label{eq:tauAstIsIotaAstEquiv}
\end{equation}
where $\iota^\ast:\Out(F_2)\to\Out(\tau(F_2))$ is the isomorphism induced by $\iota^{-1}$, and $\Out(\tau(F_2))$ acts on $\chi(W_3,G)$ via the projection $\Out(\tau(F_2))\to\Out(\tau(F_2))/\gen{[\inn_x]}\simeq\Out(W_3)$ by Equation \eqref{eq:OutW3Isom}.

Hence $\tau^\ast$ conjugates the action of $\Out(F_2)$ on $\chi(F_2,G)$ and the action of $\Out(\tau(F_2))$ on $\chi(W_3,G)$. This allows to build domains of discontinuity on $\chi(W_3,G)$ from domains of discontinuity on $\chi(F_2,G)$.

\begin{proposition}\label{prop:linkW3F2CharVar}
   Let $G$ be a group, and $\W\subset\chi(F_2,G)$ be a domain of discontinuity for the action of $\Out(F_2)$. Then $(\tau^\ast)^{-1}(\W)\subset\chi(W_3,G)$ is a domain of discontinuity for the action of $\Out(W_3)$.
\end{proposition}
\begin{proof}
   Note that $\W_0 = (\tau^\ast)^{-1}(\W)$ is open in $\chi(W_3,G)$ since $\W$ is open. Let $K\subset\W_0$ be a compact subset, and suppose that $[\phi]\in\Out(F_2)$ satisfies $\iota^\ast([\phi])\cdot K\cap K\neq\varnothing$. Then by Equation \eqref{eq:tauAstIsIotaAstEquiv}, we have after applying $\tau^\ast$ that $[\phi]\cdot\tau^\ast(K)\cap\tau^\ast(K)\neq\varnothing$. Since $\Out(F_2)$ acts properly discontinuously on $\W$ and $\tau^\ast(K)\subset\W$ is compact, there are finitely many such $[\phi]\in\Out(F_2)$, hence $\Out(\tau(F_2)) = \iota^\ast(\Out(F_2))$ also acts properly discontinuously on $\W_0$.

   Since by \eqref{eq:OutW3Isom}, $\Out(W_3)$ is a quotient of order $2$ of $\Out(\tau(F_2))$, we see that $\Out(W_3)$ also acts properly discontinuously on $\W$.
\end{proof}

\subsection{The Farey tessellation}

We now recall the definition of the Farey tessellation, and we explain its link with the (extended) conjugacy classes of primitive elements and bases of $F_2$. The \emph{extended conjugacy class} of a basis $(x,y)$ (\resp primitive element $x$) of $F_2$ is the set of bases (\resp primitive elements) of $F_2$ that are conjugated to an element of the orbit under $\gen{S,I}$ of $(x,y)$ (\resp to $x$ or $x^{-1}$). We denote by $[x,y]$ (\resp $[x]$) the extended conjugacy class of $(x,y)$ (\resp $x$).

The Farey tessellation of $\H[2]$ is the tessellation of $\H[2]$ whose vertices are the points of $\QP\subset\RP\simeq\partial\H[2]$ and with edges the biinfinite geodesics with endpoints $\{\sfrac p q,\sfrac{p'}{q'}\}$ such that $pq'-p'q = \pm 1$. It defines a tessellation of $\H[2]$ by ideal triangles (see \cite[\Ch 1]{Hatcher2022} for a detailed introduction). Fact \ref{fact:NielsenKernel} allows to define a bijection between extended conjugacy classes of primitive elements (\resp bases) of $F_2$ and the vertices (\resp edges) of the Farey tessellation. We will denote by $E$ (\resp $V$) the edges (\resp vertices) of the Farey tessellation, also called \emph{Farey edges}.

\begin{proposition}\label{prop:ECCofPrimBasisFareyEdgeCorrespondance}
   Let $\psi$ be the map sending a primitive element $u$ of $F_2$ to the ratio $\sfrac p q\in\QP = V$ where $(p,q)=\pi_e(u)$. It induces a bijection $\Psi$ between the set of extended conjugation classes of primitive elements of $F_2$ and $V$.

   Similarly, let $\tilde{\psi}$ be the map sending a basis $(u,v)$ of $F_2$ to the pair $\{\psi(u),\psi(v)\}$. Then $\tilde{\psi}$ takes values in $E$, and induces a bijection $\tilde{\Psi}$ between extended conjugacy classes of bases of $F_2$ and the set $E$ of Farey edges.
\end{proposition}
\begin{proof}
   Let $[u,v]$ be an extended conjugacy class of a basis of $F_2$. Since $\pi_e(u^{-1})=-\pi_e(u)$, $\psi(u)$ only depends on the extended conjugated class $[u]$. Hence, $\Psi$ and $\tilde{\Psi}$ are well-defined. By Equation \eqref{eq:OutF2Isom}, the matrix $(\pi_e(u),\pi_e(v))$ lies in $\GL_2(\Z)$, so it has determinant $\pm 1$, hence both $\tilde{\psi}$ and $\tilde{\Psi}$ take values in $E$. Note that $\tilde{\psi}(u,v)=\{\psi(u),\psi(v)\}$ determines the matrix $(\pi_e(u),\pi_e(v))$ up to permutation and change of sign of the columns, which amounts to right-multiplication by an element of $(\pi_e)_\ast\gen{S,I}$, so by Fact \ref{fact:NielsenKernel}, it uniquely determines $[u,v]$. Therefore, $\tilde{\Psi}$ (and thus $\Psi$ as well) is injective.

   For surjectivity, observe that given any pair of irreducible fractions $\{\sfrac p q,\sfrac{p'}{q'}\}$ such that $pq'-p'q=\pm1$ and $q,q'\geq 0$, we may suppose up to changing the roles of $(p,q)$ and $(p',q')$ that
   \begin{equation*}
      A=\begin{pmatrix}
         p & p' \\
         q & q'
      \end{pmatrix}
   \end{equation*}
   has determinant $1$. Since $\SL_2(\Z)$ is generated by the two matrices
   \begin{align*}
      (\pi_e)_\ast(L)&=\begin{pmatrix}
         1 & 1 \\
         0 & 1
      \end{pmatrix}&
      (\pi_e)_\ast(R) &= \begin{pmatrix}
         1 & 0 \\
         1 & 1
      \end{pmatrix}
   \end{align*}
   we see that there is a word $w\in\gen{L,R}<\Aut(F_2)$ such that $(\pi_e)_\ast(w) = A$. It follows that $\tilde{\Psi}([w(a,b)]) = \{\sfrac p q,\sfrac{p'}{q'}\}$, so $\tilde{\Psi}$ and hence $\Psi$ are surjective.
\end{proof}

\begin{remark}
   At the end of the proof, we could have directly used the surjectivity of $(\pi_e)_\ast$ instead, but this slight variation shows two facts. First, that an extended conjugacy class of bases of $F_2$ is determined by the choice of a unique element of $\gen{L,R}$, and second, it gives an algorithm in order to find such a word. However, we will modify this algorithm slightly so as to find words involving only positive powers of $L$ and $R$.
   
   There are other ways of computing such a word in $\gen{L,R}$ than the one we propose in the next paragraph, for instance using the continued fraction expansion of $\sfrac p q$ and $\sfrac{p'}{q'}$, as explained in \cite[\S A.3]{LeeXu2020}.
\end{remark}

\subsection{Christoffel words}

We now aim to define a family of (almost) unique representatives of extended conjugacy classes of primitive elements of $F_2$, given by \emph{Christoffel words}. We capitalize on the intuition given by the proof of Proposition \ref{prop:ECCofPrimBasisFareyEdgeCorrespondance} in order to define these.

Observe first that multiplying on the right by $(\pi_e)_\ast(L)$ (\resp $(\pi_e)_\ast(R)$) on the matrix $A$ in the proof amounts to computing the vector $(p+p',q+q')$, and replacing the right (\resp left) vector of $A$ with this new vector. Hence the \emph{Farey addition} $\sfrac p q\oplus\sfrac{p'}{q'}:=\sfrac{(p+p')}{(q+q')}$ is a reduced fraction whenever $\{\sfrac p q,\sfrac{p'}{q'}\}$ is a Farey edge. Observe moreover that it lies in the open interval with extremities $\sfrac p q$ and $\sfrac{p'}{q'}$. This allows us to make use of a dichotomy algorithm to find representatives of extended conjugacy classes of bases of $F_2$.

In order to make precise our dichotomy algorithm, we need to choose an orientation of the Farey edges. If $(r_1,r_2)$ is an oriented Farey edge, it cuts the hyperbolic space into two half-spaces. We denote by $V^+(r_1,r_2)$ (\resp $V^-(r_1,r_2)$) the subset of $\QP$ contained in the closed left (\resp right) half-space bounded by $(r_1,r_2)$. Moreover, observe that $r_1\oplus r_2$ always lies in $V^\epsilon(r_1,r_2)$ where $\epsilon\in\{\pm\}$ is the sign such that $r_1,r_2\in V^\epsilon(\infty,0)$. In the case of $(\infty,0)$
which has two representatives as pairs of reduced fractions with nonnegative denominator, we have $\pm 1 = \sfrac{\pm 1}0\oplus\sfrac 0 1\in V^\pm(\infty,0)$ depending on which representative of $\infty$ we choose.

Hence, our dichotomy algorithm for reaching $r\in\QP-\{0,\infty\}$ starts by setting $\epsilon$ to be the sign of $r$, so that $r\in V^\epsilon(\infty,0)$. Then, we define inductively a sequence $(r_1^n,r_2^n)$ of oriented Farey edges such that $r\in V^\epsilon( r_1^n,r_2^n)$ for all $n$. We initiate the sequence by setting $(r_1^0,r_2^0)=(\sfrac\epsilon 0,\sfrac01)$. Then, if $r = r_1^n\oplus r_2^n$, we stop there, and otherwise, $(r_1^{n+1},r_2^{n+1})$ is the unique oriented Farey edge in $\{(r_1^n,r_1^n\oplus r_2^n),(r_1^n\oplus r_2^n,r_2^n)\}$ such that $r\in V^\epsilon(r_1^{n+1},r_2^{n+1})$.

Since the Farey addition preserves the set of reduced fractions with nonnegative denominators, it is not difficult to prove that the number of steps in the algorithm needed to reach $r = \sfrac p q$ reduced with $q\geq 0$ is bounded above by $\max\{\abs{p},q\}$, so our algorithm terminates. Observe that if we follow along the algorithm, we can define a sequence $(x_n,y_n)$ of bases of $F_2$ such that $(\psi(x_n),\psi(y_n)) = (r_1^n,r_2^n)$ for all $n$. We start by setting $(x_0,y_0) = (a^\epsilon,b)$, and then we choose $(x_{n+1},y_{n+1})$ in $\{(x_n,x_ny_n),(x_ny_n,y_n)\} = \{L(x_n,y_n),R(x_n,y_n)\}$ mirroring the choice made for $(r_1^{n+1},r_2^{n+1})$. The representative of $\Psi^{-1}(r)$ is then the word $x_ny_n\in F_2$ where $(x_n,y_n)$ is the last basis defined in the algorithm.

This defines, for all $r\in\QP$, a unique representative $\Christoffel_e(\Psi^{-1}(r))$ in the extended conjugacy class $\Psi^{-1}(r)$, except for $r=\infty$, for which $\Christoffel_e([a])$ is non-determined, and could be either $a$ or $a^{-1}$. These words are the \emph{Christoffel words} in the basis $e$, they have the property that they are given by a unique word in the alphabet $\{L,R\}$ applied to either $(a,b)$ or $(a^{-1},b)$. Hence they always are positive words in either the basis $(a,b)$ or $(a^{-1},b)$, meaning that they only involves positive powers of $b$ and either $a$ or $a^{-1}$. In particular, Christoffel words are \emph{cyclically reduced}, meaning that they realize the minimum of the word length in their conjugacy class. Here is an example showing how the algorithm unfolds for $r=\sfrac34$.

\begin{example}
   For $r = \sfrac34$, the sequence of oriented Farey edges $(r_1^n,r_2^n)$ progressively take the values
   \begin{align*}
      &(\tfrac10,\tfrac01) &&(\tfrac11,\tfrac01) &&(\tfrac11,\tfrac12) &&(\tfrac11,\tfrac23)
   \end{align*}
   and stops, as $\sfrac11\oplus\sfrac23 = \sfrac34$. The corresponding bases $(x_n,y_n)$ are
   \begin{align*}
      &(a,b) &&(ab,b) &&(ab,ab^2) &&(ab,(ab)^2b)
   \end{align*}
   so we have $\Christoffel_e(\Psi^{-1}(\sfrac34)) = (ab)^3b$.
\end{example}

Running the same algorithm simultaneously on the two extremities of a Farey edge $\{r_1,r_2\}$ defines similarly a unique basis of $F_2$ in the extended conjugacy class of $\tilde{\Psi}^{-1}(\{r_1,r_2\})$, except for $\{\infty,0\}$ which has two representatives $(a,b)$ and $(a^{-1},b)$. Here is another example showing how the algorithm allows to find a representative of a Farey edge.

\begin{example}
   For the Farey edge $\{\sfrac{-3}4,\sfrac{-5}7\}$, the sequence of oriented Farey edges $(r_1^n,r_2^n)$ takes the values
   \begin{align*}
      &(\tfrac{-1}0,\tfrac01) &&(\tfrac{-1}1,\tfrac01)&&(\tfrac{-1}1,\tfrac{-1}2)&&(\tfrac{-1}1,\tfrac{-2}3)&&(\tfrac{-3}4,\tfrac{-2}3)
   \end{align*}
   and stops at $(\sfrac{-3}4,\sfrac{-5}7)$. The corresponding bases $(x_n,y_n)$ are
   \begin{align*}
      &(a^{-1},b)&&(a^{-1}b,b)&&(a^{-1}b,a^{-1}b^2)&&(a^{-1}b,(a^{-1}b)^2b)&&((a^{-1}b)^3b,(a^{-1}b)^2b)
   \end{align*}
   hence $((a^{-1}b)^3b,a^{-1}b((a^{-1}b)^2b)^2)$ is the basis corresponding to $\{\sfrac{-3}4,\sfrac{-5}7\}$.
\end{example}

For the rest of this paper, we will abuse the notation and identify $V$ with the set of Christoffel words in the basis $e$, and $E$ for the set of representatives of Farey edges as bases of $F_2$ given by our algorithm. As explained earlier, the correspondence only breaks for $\infty\in V$ and $\{0,\infty\}\in E$, but this should not cause any confusion later. Taking this identification further, we will also call Farey edges the bases of $F_2$ given by our algorithm. The context should be enough to distinguish between the two if needed.

We can sum up the properties of Christoffel words and Farey edges established above as follow.
\begin{enumerate}[(i)]
   \item Christoffel words and Farey edges are uniquely determined by the choice of a starting basis in $\{(a,b),(a^{-1},b)\}$ and a positive word $w\in\gen{L,R}$;
   \item Farey edges are bases of $F_2$ composed of Christoffel words;
   \item Christoffel words and Farey edges define a unique choice of cyclically reduced representatives in extended conjugacy classes of primitive elements and bases of $F_2$, except for $[a]$ and $[a,b]$.
\end{enumerate}

Lastly, we introduce a combinatorial language of which we will make abundant use throughout our proof.

\begin{definition}
   An \emph{$e$-gallery} is the choice of a starting basis among $\{(a,b),(a^{-1},b)\}$ and of a positive word in $\gen{L,R}$ (whose length will be called the length of the gallery). It can be seen on the Farey tessellation as a finite sequence of non-overlapping adjacent triangles one of which only containing the edge $\{\infty,0\}$.

   The \emph{$e$-level} $\Lv_e(u,v)$ of a basis $(u,v)$ of $F_2$ is the length of the smallest $e$-gallery ending on an element of $[u,v]$. The \emph{$e$-level} $\Lv_e(u)$ of a primitive element $u$ is the length of the shortest gallery ending at a basis whose member is an element of $[u]$.
\end{definition}

In the following, we will often omit to write the subscript $e$. It is understood that the basis is $e$ when no particular basis is mentioned.

\section{The half-length property}\label{sec:HLP}

Lee and Xu's proof relies heavily on the fact that they can associate to each basis of $F_2$ a right-angled hexagon whose side-lengths are linked to the translation lengths of the elements of the basis. We cannot make this hypothesis in general for representations $F_2\to\Isom(\H)$, so we need to assume the existence of such right-angled hexagons with loosely prescribed side-lengths. We gather these technical assumptions in what we call the half-length property.

Whenever $\rho:F_2\to\Isom(\H)$ is a representation, we will use small letters for elements of $F_2$ and the corresponding capital letter for its image by $\rho$. For instance, if $w\in F_2$, then $W = \rho(w)\in\Isom(\H)$.

\subsection{Reminders on hyperbolic geometry}
We will use several elementary results in hyperbolic geometry in the rest of the article, the interested reader may refer back to \cite{Ratcliffe2006,Martelli2022}. Throughout the article, we will assume that we are in the projective model of $\H$, seen as the projectivization of the cone of time-like vectors of the Minkowski space $\R^{d,1}$. Up to choosing an affine chart of $\RP^d$, which we endow with a fixed Euclidean metric, $\H$ can be identified with the open unit ball $B^d$ in this affine chart, and its closure $\overline\H$ with the closed unit ball $\overline{B^d}$. Among the useful features of the projective model of $\H$ is that a subspace of $\H$ is totally geodesic if and only if it is the intersection of $B^d$ with a projective subspace. Henceforth we will often often omit the adjective “totally geodesic”. The isometry group of $\H$ can be identified with the projectivization $\PO(d,1)$ of the group $\O(d,1)$ that preserves the Minkowski form of $\R^{d,1}$.

Note that up to conjugating everything by an isometry of $\H$, we can suppose without loss of generality that a given pair of orthogonal lines are diameters of $B^d$, or other similar conditions. We will frequently make such assumptions throughout the proof, in which case we will say that we are “in the \emph{configuration} where this pair of orthogonal lines are diameters of $B^d$”, for instance.

Recall also that in the projective model, one can associate to any hyperplane $H$ a point of $\RP^d-\overline{B^d}$, called its \emph{polar}. It is given by the projectivization of the orthogonal of $\Span H<\R^{d,1}$ for the Minkowski form. The following fact will be useful.

\begin{fact}[{see \cite[\Exe 6.1 (2) \& (3)]{Ratcliffe2006}}] \label{fact:polar}
   Let $H$ be a hyperplane of $\H\simeq B^d$, and let $p\in\RP^d-\overline{B^d}$ be its polar. 
   \begin{enumerate}[(i)]
      \item $p$ is the intersection of all the projective hyperplanes supporting $B^d$ at a point of $\partial B^d\cap\overline H$.
      \item A geodesic line $L\subset B^d$ is orthogonal to $H$ if and only if the projective line spanned by $L$ contains $p$.
   \end{enumerate}
\end{fact}

Another useful fact is that pairs of totally geodesic subspaces of $\H$ often admit a unique line that is perpendicular to both, called a \emph{common perpendicular}. If the two subspaces are hyperplanes that do not intersect, this can be seen as a corollary of Fact \ref{fact:polar}, as there is a unique projective line containing the polars of the two hyperplanes.

\begin{fact}[{see \cite[\Prop 2.2.3]{Martelli2022}}]\label{fact:existsUniqueCommonPerp}
   Let $E_1,E_2\subset\H$ be totally geodesic subspaces.
   \begin{enumerate}[(i)]
      \item If $\overline{E_1}\cap\overline{E_2}$ is a single point contained in $\partial\H$, then there is no line perpendicular to both $E_1$ and $E_2$.
      \item If $\overline{E_1}\cap\overline{E_2} = \varnothing$, then there is a unique line perpendicular to $E_1$ and $E_2$, and the distance $d(E_1,E_2)$ is realized exactly along this line.
   \end{enumerate}
\end{fact}

We will also need some notations for common perpendiculars and angles.

\begin{definition}\label{def:geomNotations}
   Given two geodesic subspaces $E_1$ and $E_2$ whose closures do not intersect, we denote by $\Perp(E_1,E_2)$ their unique common perpendicular. It is implicitly oriented “from $E_1$ to $E_2$”.

   Given two oriented lines $\ell_1$ and $\ell_2$, we denote by $\angle(\ell_1,\ell_2) = \angle(\ell_2,\ell_1)\in[0,\pi]$ the angle measured in the subspace spanned by $\ell_1$ and $\ell_2$, using parallel transport along $\Delta :=\Perp(\ell_1,\ell_2)$ if necessary. We will often use the notation $\angle_\Delta(\ell_1,\ell_2)$ for clarity, as a reminder that $\Delta$ is the common perpendicular of $\ell_1$ and $\ell_2$.

Lastly, we recall a few facts and definitions about isometries of $\H$. The \emph{translation length} of $\gamma\in\PO(d,1)$ is
\begin{equation*}
   \trans{\gamma}:=\inf_{o\in\H} d(o,\gamma o)
\end{equation*}
When $\trans{\gamma} = 0$, we say that $\gamma$ is \emph{elliptic} if it fixes a point in $\H$, or \emph{parabolic} otherwise. When $\trans{\gamma} >0$, we say that $\gamma$ is \emph{hyperbolic}.

A hyperbolic isometry $\gamma$ of $\H$ preserves a unique oriented geodesic line $\Ax\gamma$ in $\H$, called its \emph{axis}, on which it acts by translation by $\trans{\gamma}$ according to the orientation of $\Ax\gamma$. Hence, it fixes exactly two points in $\H$ which are the endpoints of $\Ax\gamma$. The \emph{attractive fixed point} $\gamma^+$ of $\gamma$ is the endpoint of $\Ax\gamma$ towards which it translates, and its \emph{repelling fixed point} $\gamma^-$ is the other endpoint of $\Ax\gamma$. See \cite[\S 2.2.6 \& \Cor 2.2.9]{Martelli2022} as a reference for these facts.
We will need the following normal form result for hyperbolic isometries.

\begin{fact}[{see \cite[\Thm 6.9]{Baker2002}}] \label{fact:hypIsomDec}
   Let $\gamma\in\PO(d,1)$ be a hyperbolic isometry. If we let $U<\R^{d,1}$ be the orthogonal of the subspace $\Span \gamma^- \oplus \Span\gamma^+$ for the Minkowski form, then we can write, in the decomposition $\R^{d,1} = \Span \gamma^+ \oplus U \oplus \Span \gamma^-$:
   \begin{equation*}
      \gamma = \begin{bmatrix}\mu & 0 & 0\\ 0 & O & 0\\ 0 & 0 & \mu^{-1}\end{bmatrix}
   \end{equation*}
   where $\mu > 1$, $O\in\O(q)\simeq \O(d-1)$, and $q$ is the restriction of the Minkowski form of $\R^{d,1}$ to $U$.
\end{fact}
\end{definition}

\subsection{Definition}\label{par:HLPdef}

We start by giving a framework for defining right-angled hexagons. For a pair of hyperbolic isometries $(U,V)$ such that $UV$ is hyperbolic as well, we respectively introduce $\side[\hex(U,V)]{U}$, $\side[\hex(U,V)]{V}$ and $\side[\hex(U,V)]{L}$ the common perpendiculars to the pairs of lines $(\Ax U,\Ax UV)$, $(\Ax V,\Ax UV)$ and $(\Ax U,\Ax V)$. When all three of these lines are defined, observe that they define, together with the oriented lines $\Ax U$, $\Ax V$ and $\Ax UV$, a right-angled hexagon which we denote by $\hex(U,V)$. Most of the time, we will omit the right-angled hexagon data and simply write $\side{U}$ for $\side[\hex(U,V)]{U}$, for instance. In the case of a sequence $(U_n,V_n)$, we will often write $\side{U_n}$, $\side{V_n}$ and $\side{L_n}$ (the last one stands for $\side[\hex(U_n,V_n)]{L}$).

We extend this definition to the case where two axes intersect in $\H$ ($d\geq 3$) by choosing a common perpendicular among the possible ones. This allows to define $\hex(U,V)$ as a degenerate right-angled hexagon whenever $\Ax U$, $\Ax V$ and $\Ax UV$ do not share any endpoint, and are not concurrent.
We also define an implicit ordering on $\{H_L,H_U,H_V\}$ where $\side{V} < \side{L} < \side{U}$.

Given isometries $U,V\in\Isom(\H)$ such that $\hex(U,V)$ is well-defined, we define the (signed) geometric translation length of $W\in\{U,V,UV\}$ relative to $\hex(U,V)$ as 
\begin{equation*}
\geom{\hex(U,V)}{W} = \epsilon_{XY}d(\side[\hex(U,V)]{X},\side[\hex(U,V)]{Y})
\end{equation*}
where $\side[\hex(U,V)]{X}$ and $\side[\hex(U,V)]{Y}$ are the two sides of $\hex(U,V)$ perpendicular to $\Ax W$, and $\epsilon_{XY}\in\{\pm 1\}$ is positive if the order induced by $W$ on $\{\side[\hex(U,V)]{X},\side[\hex(U,V)]{Y}\}$ coincides with the one inherited from $\hex(U,V)$, and negative otherwise.

We will also write
\begin{equation*}
   \transhex{\hex(U,V)} = \min\{\trans{U},\trans{V},\trans{UV}\}
\end{equation*}

\begin{definition}
   A representation $\rho: F_2\to\Isom(\H)$ is \emph{irreducible} if $\rho(F_2)$ has no global fixed point in $\partial\H$.
\end{definition}

\begin{definition}\label{def:HLP}
   A representation $\rho:F_2\to\Isom(\H)$ has the \emph{half-length property} if
   \begin{enumerate}
      \item $\rho$ is irreducible;
      \item the images of primitive elements are hyperbolic;
      \item there exists $\epsilon \in (0,\frac16\PrimSys(\rho)]$ such that for every Farey edge $(U,V)$:
         \begin{equation}\label{eq:HLP}
         \abs*{\frac12\trans{W} - \geom{\hex(U,V)}{W}}\leq \frac16\transhex{\hex(U,V)} - \epsilon
      \end{equation}
for all $W\in\{U,V,UV\}$.
   \end{enumerate}
\end{definition}
\begin{remark}
   In this definition, $\hex(U,V)$ is well-defined (it is possibly degenerate) for every Farey edge $(U,V)$ because of the irreducibility of $\rho$. Indeed, $\Ax U$ and $\Ax V$ cannot share an endpoint, otherwise both $U$ and $V$ would fix this point, so $\rho(F_2)$ would have a global fixed point in $\partial\H$.
\end{remark}
\begin{remark}
   If $\rho:F_2\to\Isom(\H)$ satisfies the half-length property, then for all Farey edge $(U,V)$ and $W\in\{U,V,UV\}$, $\geom{\hex(U,V)}{W}\geq\frac13\trans{W} + \epsilon > 0$. Hence it will not be necessary to discuss the sign of the geometric translation lengths later on.
\end{remark}
We denote by $\HLP\subset\chi(F_2,\Isom(\H))$ the subset of characters satisfying the half-length property. For $[\rho]\in\HLP$, we also denote by $E([\rho])$ the maximal value of $\epsilon$ such that $\rho$ satisfies \eqref{eq:HLP}. Note that both $\HLP$ and $E$ are $\Out(F_2)$-invariant.

It is not clear at first sight whether or not such representations are easy to find. However, it is not difficult to check that Coxeter extensible representations sending primitive elements to hyperbolic isometries always satisfy the half-length property.

\subsection{Coxeter extensible representations} \label{par:CoxExtReps}

A priori, Definition \ref{def:CoxExtRep} seems to depend on the choice of a basis of $F_2$. However, Proposition \ref{prop:AutF2W3Same} implies that $\rho : F_2\to\Isom(\H)$ is Coxeter extensible in the basis $(a,b)$ if and only if it is Coxeter extensible in any other basis.

Given a Coxeter extensible representation $\rho:F_2\to\Isom(\H)$, we can find three isometric involutions $I_A$, $I_L$ and $I_B$ of $\H$ satisfying $A = I_AI_L$ and $B=I_LI_B$. Since an involution in $\Isom(\H)$ is an elliptic isometry, the sets of points fixed by each of these involutions are nontrivial totally geodesic subspaces of $\H$, which we denote by $F_A$, $F_L$ and $F_B$ respectively (see \cite[\Thm 4.7.1]{Ratcliffe2006} and the discussion that follows).

 Notice that since $A = I_AI_L$, if it is hyperbolic, then $I_A\Ax A = \Ax I_LI_A = \Ax A^{-1} = \Ax A$, and similarly, $I_L\Ax A = \Ax A$, so $\Ax A$ is fixed by both $I_A$ and $I_L$. Since $I_A$ and $I_L$ are orthogonal involutions, this implies that $\Ax A$ is perpendicular to both $F_A$ and $F_L$, so we either have $F_A\cap F_L\neq\varnothing$, or $\overline{F_A}\cap\overline{F_L}=\varnothing$ by Fact \ref{fact:existsUniqueCommonPerp}. In the first case, $A$ is clearly elliptic, so $A$ hyperbolic implies $\overline{F_A}\cap\overline{F_L} = \varnothing$. If this is the case, then clearly the unique common perpendicular to $F_A$ and $F_L$ is $A$-invariant, and $A$ is hyperbolic of translation length $2d(F_A,F_L)$.

Hence $A$ is hyperbolic if and only if $\overline{F_A}\cap\overline{F_L}=\varnothing$, in which case $\Ax A$ is the unique common perpendicular to $F_A$ and $F_L$, and $d(F_A,F_L) = \frac12\trans{A}$. Similar results hold for $B = I_LI_B$ and $AB=I_AI_B$.

These properties lead to the following proposition, which exhibits a family of representations of $F_2$ satisfying the half-length property.

\begin{proposition}\label{prop:CoxExtHasHLP}
   If $\rho:F_2\to \Isom(\H)$ is Coxeter extensible and the image of every primitive element of $F_2$ is hyperbolic, then $\rho$ satisfies the half-length property.

   Moreover, $E([\rho]) = \frac16\PrimSys([\rho])$.
\end{proposition}
\begin{proof}
   First, we see that $\rho$ is irreducible. Indeed, if it wasn't, then in particular, $A$ and $B$ would preserve the same point in $\partial \H$, so $\Ax A$ and $\Ax B$ would share an endpoint. But they are both perpendicular to $F_L$, so we find a triangle in $\overline{\H}$ with angle sum $\pi$, contradicting Corollary 1 of \cite[\Thm 3.5.5]{Ratcliffe2006}.

   The discussion above shows that the lengths of the sides of $\mathcal{H}(A,B)$ along $\Ax A$, $\Ax B$ and $\Ax AB$ are as required (with $\epsilon = \frac16\PrimSys([\rho])$ clearly maximal). We conclude with Proposition \ref{prop:AutF2W3Same}.
\end{proof}

Recall that any irreducible representations of $F_2$ into $\PSL_2(\C)$ is Coxeter extensible. It turns out this is not true when the dimension is larger than $3$.

\begin{remark}\label{rmk:CoxeterExtensibleRepsDimEstimate}
   We can do a dimension estimate of the set of Coxeter extensible representations in $\chi(F_2,\Isom(\H))$, but the results depend on the dimensions of the subspaces $F_A$, $F_B$ and $F_L$ which we allow ourselves to consider. The maximal dimension happens when these three subspaces have the same dimension $\lceil\frac{d+1} 2\rceil$, giving $3\lceil \frac{d+1}{2}\rceil\lfloor\frac{d+1}{2}\rfloor - \frac{d(d+1)}{2}$. A similar estimate shows that the character variety has dimension $\frac{d(d+1)}{2}$. Therefore, the Coxeter extensible locus does not define an open subset of the character variety for $d\geq 4$.
\end{remark}

\section{Primitive stability implies the \texorpdfstring{$Q$}{Q}-conditions} \label{sec:PSinBQ}

We now recall the definitions of primitive stability and the $Q$-conditions, and give the proof of classical implications between them.

\subsection{The isometric action of \texorpdfstring{$F_2$}{F\_2} onto itself}
\label{par:Axes}

We denote by $|u|_e$ the length of a word $u\in F_2$ and by $\|u\|_e$ its cyclically reduced length.

It is a classical fact that the function $d_e(u,v) = |u^{-1}v|_e$ defines a left-invariant metric on $F_2$. Hence, the left action of $F_2$ on itself is isometric, and the translation length of a word $u\in F_2$ is $\trans[e]{u} = \|u\|_e$, which is positive whenever $u\neq 1$. We can also see that a nontrivial element $w$ of $F_2$ acts \emph{axially} in the sense of the following proposition.

\begin{proposition}\label{prop:AxisInF2}
   Choose $w\in F_2 - \{1\}$. There is a unique geodesic line $\Ax w$ on which $w$ acts as a translation by $\trans[e]{w}$. Moreover, $\Ax uwu^{-1} = u\Ax w$ for all $u\in F_2$, and when $w$ is cyclically reduced, $\{w^k,k\in\mathbb{Z}\}\subset\Ax w$.
\end{proposition}
\begin{proof}
   Remark that if $u\in F_2$, then its orbit $O_w(u)$ under the action of $w$ is $\{w^ku,k\in\mathbb{Z}\}$. Thus, $w$ acts on $O_w(u)$ with a translation length of $d_e(u,wu) = |u^{-1}wu|_e$. This translation length is minimal and equal to $\trans[e]{w}$ if and only if $u^{-1}wu$ is cyclically reduced. Now, if $u$ and $h$ are such that both $u^{-1}wu$ and $(uh)^{-1}wuh$ are cyclically reduced, then $h$ is on the unique geodesic passing through $1$ and $u^{-1}wu$.

The other two claims easily follow.
\end{proof}

\subsection{Generalizations of the \texorpdfstring{$Q$}{Q}-conditions}\label{par:generalizingBQdiscussions}

Observe that, in Definition \ref{def:BQcond}, there seems to be no canonical choice for $\lambda$ \emph{a priori}. We adopt two different ways of addressing this issue here.

\begin{definition}\label{def:strongQCond}
   A representation $\rho:F_2\to\Isom(\H)$ satisfies the \emph{strong $Q$-conditions} if it satisfies the $Q_\lambda$-conditions for all $\lambda>0$.
\end{definition}

The most general way one may want to define a $\Out(F_2)$-invariant subset of $\chi(F_2,\allowbreak\Isom(\H))$ from the $Q_\lambda$-conditions is to allow $\lambda$ to depend on $\rho$.

\begin{definition}\label{def:Qfcond}
   Let $\W\subset\chi(F_2,\Isom(\H))$ be an $\Out(F_2)$-invariant subset, and $f:\W\to\R_{>0}$ be an $\Out(F_2)$-equivariant function. A representation $\rho:F_2\to\Isom(\H)$ such that $[\rho]\in\W$ satisfies the \emph{$Q_f$-conditions} if it satisfies the $Q_{f([\rho])}$-conditions.
\end{definition}

For our purposes, we can choose $f$ such that it mostly depends on the quantity
\begin{equation*}
   \PrimSys([\rho]):=\inf\{\trans[\H]{\rho(\gamma)},\text{$\gamma\in F_2$ primitive}\}
\end{equation*}
Observe that $\PrimSys([\rho])>0$ whenever we assume that $\rho$ satisfies the $Q_\lambda$-conditions for some $\lambda>0$.

\subsection{Main theorem and known implications}

We have now introduced all the notions needed to phrase our main theorem. Note that contrary to \cite{LeeXu2020,Series2019}, our method does not allow to prove that the $Q_\lambda$-conditions are equivalent to primitive stability for some fixed $\lambda>0$ independent of the representation. Instead, as $\Lc_d\to +\infty$ as $\PrimSys([\rho])\to 0$ in the theorem below, our proof requires $\lambda$ to be large when the primitive systole is small.

\begin{theorem}\label{thm:PSisPDisBQ}
   There exists a function $\Lc_d:\HLP\to\R_{>0}$ depending only on $\PrimSys([\rho])$ and $E([\rho])$ that goes to $+\infty$ when $\PrimSys([\rho])\to 0$ such that for all representation $\rho:F_2\to\Isom(\H)$ satisfying the half-length property, the following properties are equivalent
\begin{enumerate}[(i), ref = \roman*]
\item $\rho$ is primitive stable
\item $\rho$ is primitive displacing,
\item $\rho$ satisfies the strong $Q$-conditions,
\item $\rho$ satisfies the $Q_{\Lc_d}$-conditions. 
\end{enumerate}
\end{theorem}

 Note that if we restrict to the $\Out(F_2)$-invariant subset of Coxeter extensible characters sending primitive elements to hyperbolic isometries, then $\Lc_d$ only depends on $\PrimSys([\rho])$ as said in Remark \ref{rmk:precisionMainThmCox} since $E([\rho]) = \frac16\PrimSys([\rho])$ by Proposition \ref{prop:CoxExtHasHLP}.

As noted by Minsky in \cite{Minsky2013}, it is not hard to check that primitive stability implies the primitive displacing property. This gives a proof of all the forward implications of Theorem \ref{thm:PSisPDisBQ}, the proof of the remaining implication is covered in the next section.

\begin{proposition}\label{prop:PSinPDinBQ}
   Let $X$ be a proper geodesic Gromov hyperbolic space. For representations $F_n\to\Isom(X)$, primitive stability implies primitive displacement, which implies the strong $Q$-conditions, which implies the $Q_\lambda$-conditions for all $\lambda >0$.
\end{proposition}

\begin{proof}
   Let $\rho:F_n\to\Isom(X)$ be a primitive stable representation, and $o\in X$ and $M,C>0$ be the associated point and constants given by the definition. For any $u\in V$, since $u^n\in\Ax u$ for all $n>0$:
   \begin{equation}\label{*}
      \begin{split}
         d(o,\rho(u)^no)&\geq Md_e(u^n,1)-C\\
                        &=Mn|u|_e-C\\
                        &=M n \trans[e]{u}-C
      \end{split}
   \end{equation}
   By the triangular inequality, we get after dividing by $n$ and taking the limit on both sides
   \begin{equation*}
      d(o,\rho(u)o)\geq M \trans[e]{u}
   \end{equation*}

   Observe that this inequality also holds if we replace $o$ by any $o'\in X$ since $d(o,\rho(u)^no)$ and $d(o',\rho(u)^no')$ differ by at most the constant $2d(o,o')$. Hence
   \begin{equation*}
      \trans{\rho(u)}\geq M \trans[e]{u}
   \end{equation*}
   for all $u\in V$.

   Therefore $\rho$ is primitive displacing. Notice moreover that by Equation \eqref{*}, the orbit of $o$ by $\gen{\rho(u)}$ is a quasigeodesic, so $\rho(u)$ is hyperbolic by \cite[\S 21]{GhysDelaharpe1990}.
   The rest of the implications easily follow from primitive displacement.
\end{proof}

\subsection{Representations satisfying the \texorpdfstring{$Q$}{Q}-conditions are irreducible}

The following proposition is elementary, but has its importance because it proves that all the representations we consider in this paper are irreducible. By the last proposition, it is enough to check that the $Q_\lambda$-conditions imply irreducibility. 

Notice that for $d=3$, this is a corollary of \cite[\Thm 1.4]{TanWongZhang2008EndInvariants}.

\begin{figure}[ht]
\labellist
\small\hair 3pt

\pinlabel {$p$} [tr] at 13 62
\pinlabel {$A$} [tr] at 157 47
\pinlabel {$AB$} [t] <0pt,-1pt> at 171 100
\pinlabel {$B$} [tl] at 108 184

\pinlabel {$AB \mathcal{H}$} [b] at 73 161
\pinlabel {$\mathcal{H}$} [b] at 55 131
\pinlabel {$B \mathcal{H}$} [b] <2pt,0pt> at 32 95

\endlabellist
\centering
\includegraphics[scale=0.7]{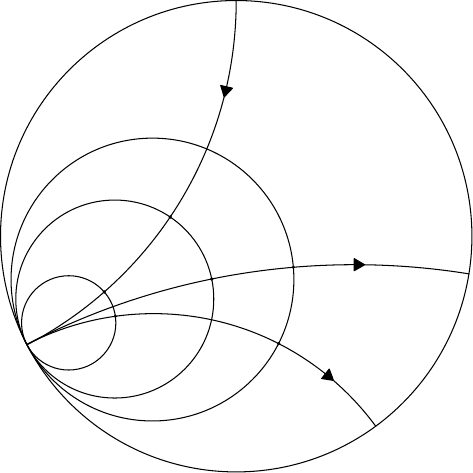}
   \caption{\emph{(Poincaré model of $\H$)} The axes of the pair $(AB,B)$ are in the same configuration as those of the pair $(A,B)$, but $\trans{AB} < \trans{A}$. Iterating gives a contradiction with the $Q$-conditions}
\label{fig:irreducibility}
\end{figure}

\begin{proposition}\label{prop:BQisIrr}
    For any $\lambda>0$, the $Q_\lambda$-conditions imply irreducibility.
\end{proposition}

\begin{proof}
   Suppose $\rho:F_2\to\Isom(\H)$ is not irreducible. We find $p\in\partial\H$ such that $p$ is fixed by $A$ and $B$ (see Figure \ref{fig:irreducibility}). Suppose first that $p=A^+=B^-$ or $p=A^-=B^+$.

   Choose $\mathcal{H}$ an horosphere centered at $p$. We see that $B\mathcal{H}$ and $AB\mathcal{H}$ are two other horospheres centered at $p$. Moreover, $d(\mathcal{H},B\mathcal{H})=\trans{B}$, $d(\mathcal{H},AB\mathcal{H})=\trans{AB}$ and $d(B\mathcal{H},AB\mathcal{H})=\trans{A}$. This yields
   \begin{equation*}
      \trans{AB}=\abs{\trans{A}-\trans{B}}<\max\{\trans{A},\trans{B}\}
   \end{equation*}

Since we supposed that $p=A^+=B^-$ or $p=A^-=B^+$, we can find $U\in \{A,B\}$ such that $p=(AB)^+=U^-$ or $p=(AB)^-=U^+$. We iterate this with the Farey edge $(A,AB)$ or $(AB,B)$ depending on which choice of $U$ was needed, and get a sequence of Christoffel words $u_n\in V$ such that $\trans{\rho(u_n)}$ is strictly decreasing. Moreover, this sequence is easily identified with the sequence of remainders in the subtraction algorithm for computing the greatest common divider of $\trans{A}$ and $\trans{B}$. Thus, either it reaches $0$ (in which case we have found a primitive element whose image is not hyperbolic), or the sequence goes to $0$. Hence, $\rho$ cannot satisfy both $Q_\lambda$-conditions at once.

   In the case where $p=A^+=B^+$ or $p=A^-=B^-$, the same argument applies by considering the basis $(A^{-1},B)$ instead.
\end{proof}

\section{The \texorpdfstring{$Q$}{Q}-conditions imply primitive stability}\label{sec:proofOfConverse}

In their proof, Lee and Xu used geometric arguments that can easily be generalized to upper dimensions (provided that the representations satisfy the half-length property), except for one lemma that makes use of trigonometric rules for right-angled hexagons, for which no equivalent is known in dimension higher than $4$ (in dimension $4$, Tan, Wong and Zhang found appropriate formulae, and in dimension $5$, work of Delgove and Retailleau suggests there may be similar formulae \cite{TanWongZhang2012,DelgoveRetailleau2014}). We will follow their proof and provide a geometric argument in order to complete it in upper dimensions.
The exact statement we prove is
\begin{theorem}\label{thm:BQ+HLinPS}
   There exists a function $\Lc_d:\HLP\to\R_{>0}$ depending only on $\PrimSys([\rho])$ and $E([\rho])$ that goes to $+\infty$ when $\PrimSys([\rho])\to 0$ such that any representation $\rho:F_2\to\Isom(\H)$ satisfying the half-length property and the $Q_{\Lc_d}$-conditions is primitive stable.
\end{theorem}

In fact, the proof would allow one to give an explicit formula for $\Lc_d$ in terms of $\PrimSys$ and $E$, although neither optimal nor enlightening. We do not write such a formula here for these reasons.

Every time a representation $\rho:F_2\to\Isom(\H)$ is mentioned in the rest of this section, it is assumed to satisfy the assumptions of Theorem \ref{thm:BQ+HLinPS}.
We do not yet have the means to explain the conditions $\Lc_d$ needs to satisfy, so we will make them clear \emph{during the proof}, through statements like “if $\Lc_d([\rho])$ is large enough”.

\begin{remark}\label{rmk:defLvN}
   If $\rho$ satisfies the $Q_{\Lc_d}$-conditions, then there are finitely many $[x]\in V$ such that $\trans{\rho(x)} < \Lc_d([\rho])$. Therefore, there is a level $N$ such that, for all $[x]\in V$ of level at least $N$, $\trans{\rho(x)}\geq\Lc_d([\rho])$. In what follows, $N$ will refer to such an integer.
\end{remark}

\subsection{Preliminary results}\label{par:TecLem}

We start with some useful notations and technical results. See Figure \ref{fig:tecNotations} for reference.

\begin{figure}[ht]
   \labellist
   \small\hair 2pt
   \pinlabel $\textcolor{heavyblue}{V}$ [l] at 179 160
   \pinlabel $\textcolor{heavyred}{U}$ [r] at 113 103
   \pinlabel $\textcolor{heavygreen}{UV}$ [l] at 119 129

   \pinlabel $\side{L}$ [b] at 77 118
   \pinlabel $\side{V}$ [tr] at 172 35
   \pinlabel $\side{U}$ [t] at 58 193

   \pinlabel $\textcolor{heavyblue}{\hyp{L}{V}}$ [b] at 41 156
   \pinlabel $\textcolor{heavyblue}{\hyp{V}{V}}$ [br] <3pt,-1pt> at 60 31
   \pinlabel $\textcolor{heavyred}{\hyp{L}{U}}$ [t] <0pt,-1pt> at 38 95
   \pinlabel $\textcolor{heavyred}{\hyp{U}{U}}$ [t] at 152 190
   \pinlabel $\textcolor{heavygreen}{\hyp{V}{UV}}$ [l] at 67 28
   \pinlabel $\textcolor{heavygreen}{\hyp{U}{UV}}$ [br] at 35 195
   \endlabellist
   \centering
   \includegraphics[scale=.70]{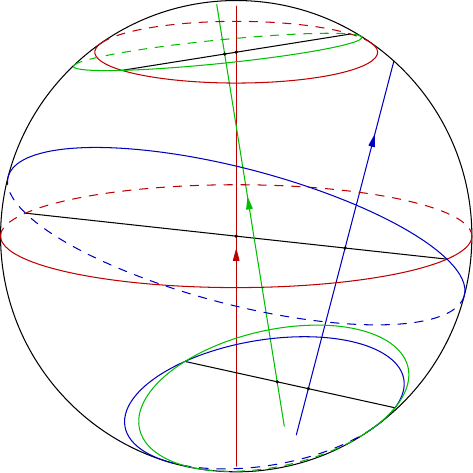}
   \caption{The right-angled hexagon $\hex(U,V)$ decorated by its orthogonal hyperplanes. A plane of the same color as the axis of an isometry is orthogonal to this axis. The black lines are contained in the intersection of the two hyperplanes around it}
   \label{fig:tecNotations}
\end{figure}

\begin{definition} \label{def:tecNotations}
   For a given Farey edge $(U,V)$, and a pair $(s,t)\in \{U,V,L\}\times\{U,V,UV\}$ such that $\side{s}$ meets $\Ax t$ orthogonally, we denote by
\begin{itemize}
\item $\hyp{s}{t}$ the hyperplane orthogonal to $\Ax t$ that contains $\side{s}$;
\item $\hem{s}{t}$ the closed half-space whose boundary is $\hyp{s}{t}$ that does not contain $\side{s'}$ if $\side{s'}$ is the other perpendicular to $\Ax t$;
\item $\ohem{s}{t} = \hem{s}{t}-\hyp{s}{t}$ the open half-space that is the interior of $\hem{s}{t}$.
\end{itemize}
\end{definition}

Here is an elementary lemma of which we will make repeated use in what follows.

\begin{figure}[ht]
   \centering
   \begin{subfigure}[t!]{.45\textwidth}
      \labellist
      \pinlabel $\Delta$ [b] at 165 108
      \pinlabel $\P$ [t] at 162 92
      \pinlabel $\P'$ [b] at 172 151
      \pinlabel $\Q$ [bl] <-1pt,-1pt> at 70 45
      \endlabellist
      \includegraphics[scale=.70]{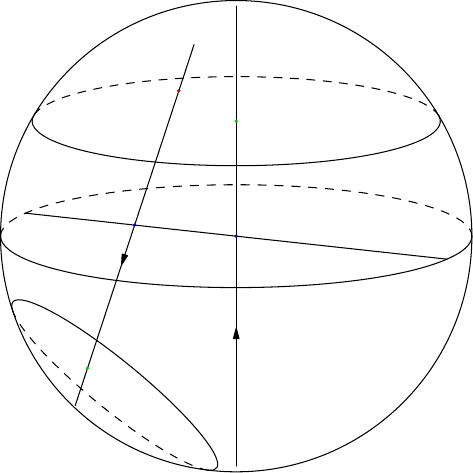}
      \hfill
      \caption{}
      \label{fig:hypNotInt1}
   \end{subfigure}
   \begin{subfigure}[t!]{.45\textwidth}
      \hfill
      \labellist
      \pinlabel $\Delta$ [b] at 165 108
      \pinlabel $\P$ [t] at 162 92
      \pinlabel $\P'$ [b] at 172 151
      \pinlabel $\Q$ [l] at 64 62
      \endlabellist
      \includegraphics[scale=.70]{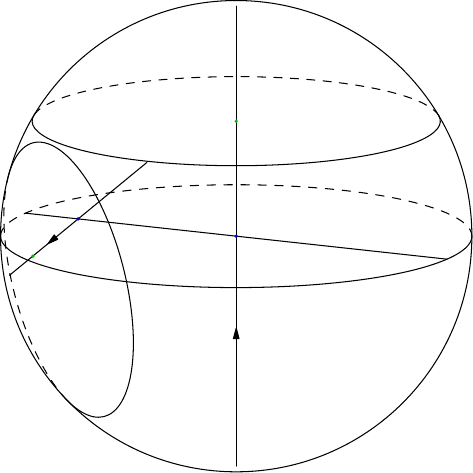}
      \caption{}
      \label{fig:hypNotInt2}
   \end{subfigure}
   \caption{In these two configurations, we can guarantee that some hyperplanes do not intersect}
   \label{fig:hypNotInt}
\end{figure}

\begin{lemma}\label{lm:hypNotInt}
For any $\mu >0$, there exists some $R(\mu)>0$ such that for every hyperplanes $\P$, $\P'$ and $\Q$, and for every line $\Delta\subset\P$, if
\begin{equation*}
\left\{\begin{array}{l}d(\P,\P') = d(\Delta,\P') \geq \mu\\d(\Delta,\Q)\geq R(\mu)\end{array}\right.
\end{equation*}
then
\begin{enumerate}[(i), ref = \roman*]
   \item If $\Perp(\Delta,\Q)\cap \P' \neq\varnothing$, then $\P\cap\Q = \varnothing$ (see Figure \ref{fig:hypNotInt1}); \label{item:lm1}
   \item If $\angle_\Delta(\Perp(\Delta,\Q),\Perp(\Delta,\P'))>\frac{\pi}{2}$, then $\P'\cap \Q = \varnothing$ (see Figure \ref{fig:hypNotInt2}). \label{item:lm2}
\end{enumerate}
\end{lemma}

\begin{proof}
   We suppose that we are in the configuration depicted in Figure \ref{fig:hypNotInt}, where $\Delta$ and $\Perp(\Delta,\P')$ are diameters of $B^d$.
   Let $\Pi$ be the projective subspace spanned by $\Delta$ and $\Perp(\Delta,\Q)$, we will denote by $B^2$ the intersection $B^d\cap\Pi$, and by $\Cc$ the circle $\partial B^2$. We claim that the Euclidean diameter of $\Q$ in $B^d$ is the length of the interval $\Q\cap B^2$. Indeed, since $\Q$ is a ball obtained by intersecting $B^d$ with a hyperplane, its diameter is the length of its intersection with any $2$-dimensional subspace containing the Euclidean center $O$ of $B^d$, which is the case of $\Pi$ since we assumed that $\Delta\subset\Pi$ is a diameter of $B^d$.

\begin{figure}[ht]
   \labellist
   \small\hair 3pt
   \pinlabel $\Cc$ <2pt,-1pt> [br] at 30 149
   \pinlabel $\Delta$ <0pt,1pt> [t] at 47 95
   \pinlabel $x$ <-2pt,0pt> [b] at 224 95
   \pinlabel $y$ <0pt,-2pt> [l] at 95 188
   \pinlabel $\Perp(\Delta,\Q)$ <1.5pt,0pt> [r] at 159 63
   \pinlabel $\Q$ <1.5pt,1.5pt> [tr] at 118 156
   \pinlabel $O$ <1pt,0pt> [tr] at 95 95
   \pinlabel $P$ <2pt,0pt> [t] at 205 95
   \pinlabel $M$ <0.5pt,-1pt> [tr] at 159 95
   \pinlabel $N$ <-1.5pt,0pt> [bl] at 159 150
   \pinlabel $J$ [tr] at 159 128
   \pinlabel $J_1$ <2pt,1pt> [tr] at 177 116
   \pinlabel $J_2$ [t] at 86 179
   \endlabellist
   \centering
   \includegraphics[scale=.80]{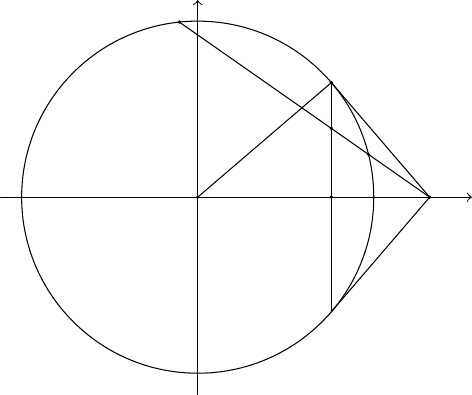}
   \caption{Let $P$ denote the polar of $\Perp(\Delta,\Q)$ in $B^2$. The other points are defined on the diagram. In the right triangle $ONM$, Pythagoras's theorem gives $\norm{OM} = \sqrt{1-(d/2)^2}$. Since the triangles $ONM$ and $ONP$ are similar, we have $\norm{OP} = 1/\sqrt{1-(d/2)^2}$}.
   \label{fig:eucDiamOfHypComputation}
\end{figure}

   Therefore, it is enough to work in $B^2$ and find the Euclidean length of the segment $\overline{\Q\cap B^2}$. We take the notations introduced in Figure \ref{fig:eucDiamOfHypComputation}, and we let $d$ be the Euclidean length of $\Perp(\Delta,\Q)$ and $h$ be the Euclidean length of the segment $JN$. We also choose Euclidean coordinates such that $O$ is the origin, $\Delta$ is the $x$-axis, and $B^2$ is the unit disk. Elementary euclidean geometry (as explained in the caption of Figure \ref{fig:eucDiamOfHypComputation}) give the coordinates of the points $P$ and $J$, which allows to find the coordinates of $J_1$ and $J_2$ by finding for which $t\in\R$ does the point $P + t PJ$ lie in the circle $\Cc$. If we let $t_1$ and $t_2$ be the solutions of the associated quadratic equation, we want to compute the length of
   \begin{equation*}
   J_1J_2 = \frac{\frac d 2 (t_2 - t_1)}{\sqrt{1 - \frac {d^2} 4}}\left(-\frac d 2, \left(1-\frac{2d} h\right)\sqrt{1 - \frac {d^2} 4}\right)
   \end{equation*}
   After solving the equation for $t_2-t_1$ and simplifying, we get
   \begin{align*}
      \norm{J_1J_2} &= 2\sqrt{\frac{\frac{2h}d\left(2-\frac{2h}2\right)}{1+\left(\frac{4}{d^2}-1\right)\left(1-\frac{2h}{d}\right)^2}}
      \leq 2\sqrt{\frac{2h}d\left(2-\frac{2h}2\right)}\\
      &\leq 4\sqrt{\frac{h}{d}}
   \end{align*}

   Since $d(\Delta,\Q)=\frac12\log(\frac d h - 1)$, it follows that the Euclidean diameter of $\Q$ in this configuration is smaller than $4e^{-d(\Delta,\Q)}$. Therefore, the Euclidean distance between the forward endpoint of $\Perp(\Delta,\Q)$ and the boundary of $\Q$ is smaller than $4e^{-d(\Delta,\Q)} + 2e^{-2d(\Delta,\Q)}$.

   Each of the two cases of the lemma ensures that the Euclidean distance between the forward endpoint of $\Perp(\Delta,\Q)$ and $\P$ (in the first case) or $\P'$ (in the second case) is at least $\tanh \mu$. Therefore, choosing $R(\mu)$ such that $4e^{-R(\mu)} + 2e^{-2R(\mu)} < \tanh \mu$ is enough.
\end{proof}

The following is a technical lemma in hyperbolic geometry that will be useful later on. The proof of the general case was kindly suggested by Konstantinos Tsouvalas.

\begin{figure}[ht]
   \labellist
   \small\hair 3pt
   \pinlabel $g^+$ by .5 1 at 61 14
   \pinlabel $g^-$ by -1 0.3 at 204 47
   \pinlabel $h^{-1}g^+$ by 1 .4 at 4 87

   \pinlabel $o$ by -1.5 0 at 81 40
   \pinlabel $o'$ by 1 -.4 at 66 45
   \pinlabel $ho'$ by -.3 -1 at 183 59

   \pinlabel $g$ [b] at 125 63
   \pinlabel $h$ [b] at 112 114
   \pinlabel $h^{-1}gh$ [bl] at 56 68
   \endlabellist
   \centering
   \includegraphics[scale=.70]{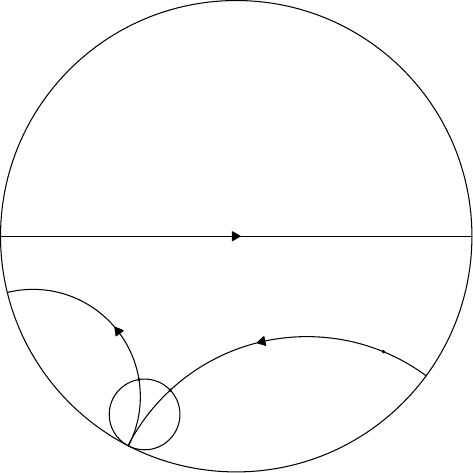}
   \caption{\emph{(Poincaré model of $\H$)} If $hg^+ = g^-$, then if we choose $o = o_n$ and the corresponding $o'_n$ so that $g^nho'_n = o_n$, we see that $\trans{g^nh}\to 0$}
   \label{fig:transLength}
\end{figure}

\begin{lemma}\label{lm:transLength}
If $g,h\in\Isom(\H)$ are hyperbolic, then there exists $t\in\R$ such that
\begin{equation*}
\trans{g^nh} = n\trans{g} + t + o(1)
\end{equation*}
unless $hg^+ = g^-$, in which case $\trans{g^nh}\to 0$.
\end{lemma}

\begin{proof}
We start with the case where $hg^+ = g^-$, in which case we have that $\Ax g$ and $h^{-1}\Ax g$ share $g^+$ as an endpoint. Thus, if we denote by $\hor_o$ (for $o\in\Ax g$) the horosphere centered at $g^+$ passing through $o$, then $d(o', o) \to_{o\to g^+} 0$, if $o' = \hor_o\cap h^{-1}\Ax g$.

   Define $f(o) = d(o,ho')$. This is a non-negative function $\R\to\R_+$ that strictly decreases, reaches $0$ exactly at one point $o_0$, and then strictly increases (up to choosing a parametrization of $\Ax g$ compatible with its orientation). Since $f\to_{g^{\pm}}\infty$, after restriction to the geodesic ray $[o_0,g^+)$, $f$ becomes an increasing bijection onto $\R_+$, and for every $n$, there exists a unique $o_n\in[o_0,g^+)$ such that $d(o_n,ho'_n) = n\trans{g}$. Observe that $o_n\to g^+$ since $f(o_n) = n\trans{g}\to +\infty$.

Thus, $g^nho'_n = o_n\in\hor_{o_n}$. Hence, $\trans{g^nh}\leq d(o'_n,o_n)\to 0$.

\medskip

Suppose now that $hg^+ \neq g^-$. As in Fact \ref{fact:hypIsomDec}, we let $U$ be the orthogonal in $\R^{d,1}$ of $\Span g^+\oplus\Span g^-$, so that, in the decomposition $\R^{d,1} = \Span g^+\oplus U\oplus \Span g^-$, we have
   \begin{equation*}
      g = \begin{bsmallmatrix}\mu &0&0\\0&O&0\\0&0&\mu^{-1}\end{bsmallmatrix}
   \end{equation*}
for some $\mu > 1$ and $O\in\O(d-1)$.

Now, since $hg^+\neq g^-$, we can also write $h$ as
\begin{equation*}
   h = \begin{bmatrix}\sigma & w_1\\w_2& M\end{bmatrix} \in \PO(d,1)
\end{equation*}
in the decomposition $\R^{d,1} = \Span g^+ \oplus (U\oplus\Span g^-)$, where $\sigma \geq 0$ (up to changing the sign of $\sigma$, $w_1$, $w_2$ and $M$). Notice moreover that we must have $\sigma > 0$, as if $\sigma = 0$, then $hg^+\in \Proj(U\oplus\Span g^-)\cap\partial B^d = \{g^-\}$, so $hg^+ = g^-$, contradicting our assumption.

   By Fact \ref{fact:hypIsomDec}, a hyperbolic isometry $\gamma\in\PO(d,1)$ has a unique maximal positive eigenvalue $\mu_1(\gamma) > 0$. It is a classical result that $\trans{\gamma} = \log\mu_1(\gamma)$ (see for instance \cite[\Lm 2.8]{CooperDelp2010}).

   We have, after letting $N = \begin{psmallmatrix}O&0\\0&\mu^{-1}\end{psmallmatrix}$:
   \begin{equation} \label{eq:renormSeqConverges}
      \begin{aligned}
         \mu^{-n} g^nh &= \mu^{-n} \begin{bmatrix}\mu & 0\\ 0 & N\end{bmatrix}^n \begin{bmatrix}\sigma & w_1\\ w_2 & M\end{bmatrix}
         = \begin{bmatrix}\sigma & w_1\\ \mu^{-n}N^n w_2 & \mu^{-n}N^nM\end{bmatrix} \\
         &\tosub{n\to\infty} \begin{bmatrix}\sigma & w_1 \\ 0 & 0\end{bmatrix}
      \end{aligned}
   \end{equation}

Now, observe that if we choose $(v_n)_n$ a sequence of eigenvectors of $g^nh$ for the eigenvalues $\mu_1(g^nh)$ of unit length for a fixed scalar product, we can choose a subsequence $(v_{\theta(n)})_n$ that converges to a unit length vector $v_\infty$. Since we have
\begin{equation*}
   \mu^{-n}g^nhv_n = \mu^{-n}\mu_1(g^nh)v_n = \mu_1(\mu^{-n} g^nh)v_n
\end{equation*}
the left hand side has a limit by Equation \ref{eq:renormSeqConverges}, and since we also have $v_{\theta(n)}\to v_\infty\neq 0$, we deduce that $\mu_1(\mu^{-\theta(n)} g^{\theta(n)}h)$ also admits a limit. Hence, if we write $v_{\infty} = \begin{bsmallmatrix}e\\f\end{bsmallmatrix}$, we obtain the following identity after taking the limit:
\begin{equation*}
   \begin{bmatrix}\sigma & w_1\\0&0\end{bmatrix}\begin{bmatrix}e\\f\end{bmatrix} = \lim_n\mu_1(\mu^{-\theta(n)}g^{\theta(n)}h)\begin{bmatrix}e\\f\end{bmatrix}
\end{equation*}
Therefore, $f = 0$ and $\lim_n\mu_1(\mu^{-\theta(n)}g^{\theta(n)}h) = \sigma$.
Since this limit does not depend on the choice of the converging subsequence, we obtain that $\mu^{-n}\mu_1(g^nh)\to \sigma$, which gives the result after taking the logarithm and noticing that $\mu = \mu_1(g)$.
\end{proof}

\subsection{Narrow Farey edges}

In the following, we will frequently need to apply Lemma \ref{lm:hypNotInt} to configurations coming from Farey edges, so we will assume that $\Lc_d([\rho])\geq 6R(\frac16\PrimSys([\rho]))$. Since $\rho$ satisfies the $Q_{\Lc_d}$-conditions, this ensures by Remark \ref{rmk:defLvN} that we can apply Lemma \ref{lm:hypNotInt} with certain Farey edges of level at least $N$ satisfying additional geometric requirements. This is our first constraint on $\Lc_d$, but we will later have to assume it is larger in Proposition~\ref{prop:finitenessNonAcute}.

\begin{definition} \label{def:narrow}
A Farey edge $(U,V)$ is said to be \emph{narrow} if one of the following is satisfied:
\begin{itemize}
\item $U^+\in \ohem{L}{V}$
\item $V^-\in \ohem{L}{U}$
\end{itemize}
\end{definition}
\begin{remark}
It is easy to see that these two conditions are equivalent. Indeed, a Farey edge $(U,V)$ is narrow if and only if $\angle_{\side{L}}(\Ax U,\Ax V) < \frac{\pi}{2}$.
\end{remark}

We start by making some elementary observations about narrow Farey edges. The first one is that if a narrow basis has a large enough level, then both its successors in the Christoffel construction are narrow Farey edges.

\begin{lemma}
Let $(U,V)$ be a Farey edge. If $\side{V}\cap\hem{U}{U}=\varnothing$ then $L(U,V)$ is narrow. Similarly, if $\side{U}\cap\hem{V}{V}=\varnothing$, then $R(U,V)$ is narrow.
\end{lemma}
\begin{proof}
   Suppose $\side{V}\cap\hem{U}{U}=\varnothing$. Write $(U',V') = L(U,V) = (U,UV)$. We know that $\Ax UV$ goes from $\side{V}$ to $\side{U}\subset\hyp{U}{U}$, which implies that $(UV)^+ \in \hem{U}{U}$ by assumption. Now, $\side{L'} = \side{U}$, and $U' = U$, so $\ohem{L'}{U'} = \H-\hem{U}{U}$. Thus, $V'^-\in\allowbreak\ohem{L}{U'}$ \ie $(U',V')$ is narrow. The proof is the same for the second assertion.
\end{proof}

As a corollary, we obtain

\begin{corollary}\label{cor:narrowIf}
Let $(U,V)$ be a Farey edge. If $\hem{U}{U}\cap\hem{V}{V}=\varnothing$, then both $L(U,V)$ and $R(U,V)$ are narrow.
\end{corollary}
In addition, we have
\begin{proposition}\label{prop:heredityThreshold}
For all Farey edge $(U,V)$ with $\Lv (U,V)\geq N$, if $(U,V)$ is narrow, then $\hem{U}{U}\cap\hem{V}{V}=\varnothing$.
\end{proposition}
\begin{proof}
Suppose that $\Lv V = \Lv (U,V)$, the other case being symmetrical. By the half-length property, $d(H_L,\hyp{U}{U}) = d(\hyp{L}{U},\hyp{U}{U}) = \geom{\hex(U,V)}{U}\geq\frac13\trans{U} + \epsilon\geq \frac16\PrimSys([\rho])$ and $d(\hyp{V}{V},H_L) = \geom{\hex(U,V)}{V}\geq\frac13\trans{V} + \epsilon \geq R(\frac16\PrimSys([\rho]))$, where the last inequality follows from Remark \ref{rmk:defLvN} and the fact that $\Lv V\geq N$. Since moreover $(U,V)$ is narrow, $\angle_{H_L}(\Perp(H_L,\hyp{U}{U}),\Perp(H_L,\hyp{V}{V})) = \pi - \angle_{H_L}(\Ax U,\Ax V) > \pi/2$, so the conclusion follows from Lemma \ref{lm:hypNotInt}.\eqref{item:lm2} applied with $(\Delta,\P,\P',\Q) = (\side{L},\hyp{L}{U},\hyp{U}{U},\hyp{V}{V})$.
\end{proof}
As a result, we get
\begin{corollary}\label{cor:heredityNarrow}
If a Farey edge $(U,V)$ of level greater than $N$ is narrow, then all its successors in the Farey triangulation are narrow.
\end{corollary}

The rest of the proof is more technical. We will need the following lemma about right-angled hexagons that are very far from being narrow.

\begin{lemma}\label{lm:additivityForRHHfarFromNarrow}
   For all $\mu>0$, there is $K(\mu) >0$ such that for all Farey edge $(U,V)$, if
   \begin{align*}
      d(\Ax S,\Ax UV) &\leq K(\mu) & \angle_{H_S}(\Ax S,\Ax UV) &\geq \pi-K(\mu)
   \end{align*}
   for some $S\in\{U,V\}$, then
   \begin{equation*}
      \abs{\geom{\hex(U,V)}{UV} + \geom{\hex(U,V)}{S}-\geom{\hex(U,V)}{T}}\leq \mu
   \end{equation*}
   where $T\in\{U,V\}$ satisfies $\{S,T\} = \{U,V\}$.
\end{lemma}
\begin{proof}
   We will do the proof in the case where $S = U$, the other one being similar. We place ourselves in a configuration where $\Ax U$ and $H_U$ are diameters of $B^d$ (see Figure \ref{fig:finitenessNonAcute2} for all the notations and estimates below). 

   Up to choosing $K(\mu)$ small, we may suppose that $\Ax U$ and $\Ax UV$ are as close to each other as we want for the Euclidean Hausdorff distance, with opposite orientation. In particular, we can also suppose up to choosing $K(\mu)$ small enough that $\hyp{U}{U}$ and $\hyp{U}{UV}$ are as close as we need for the Euclidean Hausdorff distance, since they are respectively orthogonal to $\Ax U$ and $\Ax UV$, and both contain $H_U$. Similarly, $\Ax V = \Perp(H_V,H_L)$ can be assumed to be as closed to $\Ax UV$ as needed for the Euclidean Hausdorff distance on $B^d$, with the same orientation, up to choosing $K(\mu)$ small enough, as $H_L$ and $H_V$ are as closed to being Euclideanly perpendicular to $\Ax U$ as needed up to choosing $K(\mu)$ small enough. Thus, if we let $o = \Ax V\cap \hyp{U}{U}$, $o' = \Ax V\cap \hyp{U}{UV}$, $k = \Ax U\cap \side{U}$ and $l = \Ax UV\cap \side{U}$, we see that $d(o,o')$ and $d(k,l)$ can be assumed to be as small as needed, up to shrinking $K(\mu)$. We will assume that $d(o,o')$ and $d(k,l)$ are at most $\mu$ in what follows.
   
   Now, if we let $p=\Ax V\cap\side{V}$, $p' = \Ax UV\cap \side{V}$, $q=\Ax V\cap\side{L}$ and $q'=\Ax U\cap\side{L}$, we have:
\begin{align*}
   \geom{\hex(U,V)}{V} &= d(\side{V},\side{L})\leq d(p',q')\\
   &\leq d(p',l) + d(q',k) + d(k,l)\\
   &\leq \geom{\hex(U,V)}{UV} + \geom{\hex(U,V)}{U} + \mu\\
   \shortintertext{and}
   \geom{\hex(U,V)}{V}&= d(p,q)\geq d(p,o')+d(q,o)-d(o,o')\\
   &\geq d(\side{V},\hyp{U}{UV})+d(\side{L},\hyp{U}{U})-d(o,o')\\
   &\geq \geom{\hex(U,V)}{UV}+\geom{\hex(U,V)}{U}-\mu
\end{align*}
so that $\abs{\geom{\hex(U,V)}{UV}+\geom{\hex(U,V)}{U}-\geom{\hex(U,V)}{V}}\leq\mu$.
\end{proof}

\begin{figure}[ht]
   \centering
   \begin{subfigure}[t]{.45\textwidth}
      \hfill
      \labellist
      \small\hair 1.5pt
      \pinlabel $U$ <2pt,0pt> [l] at 113 127
      \pinlabel $UV$ <0pt,-1pt> [r] at 77 165
      \pinlabel $V$ <2pt,-5pt> [l] at 94 80

      \pinlabel* $\side{L}$ <1pt,2pt> [tr] at 43 26
      \pinlabel $\side{U}$ [t] at 143 110
      \pinlabel* $\side{V}$ <-1pt,-1pt> [tl] at 88 181

      \pinlabel $\hyp{L}{U}$ [b] at 163 43
      \pinlabel $\hyp{V}{V}$ <0pt,-2pt> [l] at 134 210
      \pinlabel $\hyp{U}{U}$ [tr] at 190 132
      \pinlabel $\hyp{U}{UV}$ <-1.5pt,-1pt> [bl] at 41 68
      \pinlabel {$\scriptstyle k$} <1pt,-1pt> [tl] at 113 114
      \pinlabel {$\scriptstyle l$} [tr] at 83 117
      \pinlabel {$\scriptstyle o'$} <1pt,4pt> [tl] at 89 107
      \pinlabel {$\scriptstyle o$} [l] at 89 112
      \pinlabel {$\scriptstyle q'$} [bl] at 113 31
      \pinlabel {$\scriptstyle q$} [bl] at 102 30
      \pinlabel {$\scriptstyle p$} [bl] at 75 196
      \pinlabel {$\scriptstyle p'$} [tr] at 73 198
      \endlabellist
      \includegraphics[scale=.70]{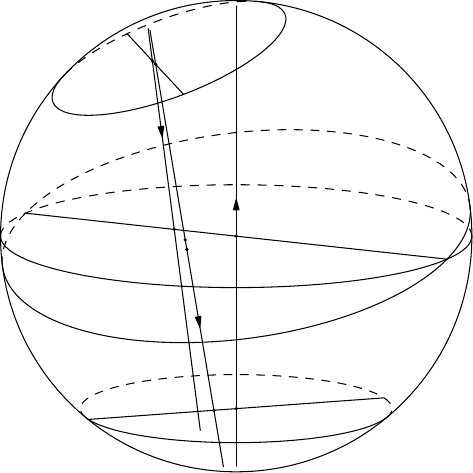}
      \caption{Both $\Ax V$ and $\Ax UV$ are close to $\Ax U$ for the Euclidean Hausdorff metric. So, $p$ and $p'$, $q$ and $q'$, as well as $o$, $o'$, $k$ and $l$ are close to each other}
      \label{fig:finitenessNonAcute2}
   \end{subfigure}
   \hspace{.03\textwidth}
   \begin{subfigure}[t]{.45\textwidth}
      \labellist
      \pinlabel $\P_{n,k}$ at 145 200
      \pinlabel $\side{L_n}$ <0pt,-2pt> [b] at 197 105
      
      \pinlabel $U$ [l] at 113 156
      \pinlabel $V_n$ [r] at 70 151
      \endlabellist
      \includegraphics[scale=.70]{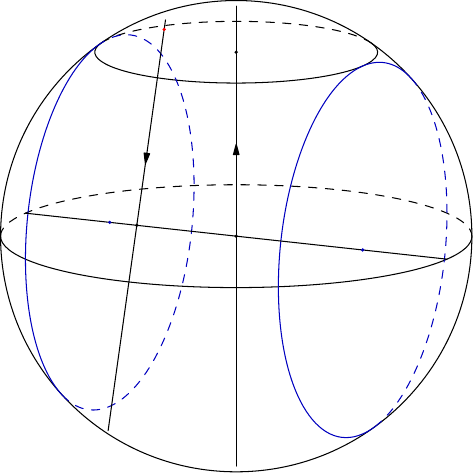}
      \hfill
      \caption{The blue circles indicate the orthogonal projection of $\P_{n,k}$ onto $H_{L_n}$. The Euclidean diameter of this projection is $2K_k$ and only depends on $k$}
      \label{fig:LnRn}
   \end{subfigure}
   \caption{}
\end{figure}

We also know that going only left or right leads us to narrow Farey edges.
\begin{proposition}\label{prop:LnRn}
For any Farey edge $(U,V)$, $L^n(U,V)$ and $R^n(U,V)$ are narrow for $n$ large enough.
\end{proposition}
\begin{proof}
We only do the proof for $L^n(U,V) = (U,U^nV)=:(U_n,V_n)$, the other case being similar. We will denote by $\hex_n$ the corresponding right-angled hexagons in this proof.

By Lemma \ref{lm:transLength}, we must have $\trans{V_n} = n\trans{U} + t + o(1)$ for some $t\in\R$ since $\trans{V_n}\to 0$ would violate the $Q$-conditions. Hence, the half-length property implies that both $\geom{\hex_n}{V_n}$ and $\geom{\hex_n}{V_{n+1}}$ diverge to $+\infty$.

Observe that in a configuration where $\side{V_n}$ and $\Ax V_n$ are diameters of $B^d$, the last observation and the fact that $\geom{\hex_n}{U}$ is bounded above (by the half-length property) shows that $d(\Ax V_n,\Ax V_{n+1}) \to 0$. 
Indeed, if there was $\kappa >0$ such that $d(\Ax V_n,\Ax V_{n+1}) \geq\kappa$ for infinitely many $n$, then we would find a uniform $\kappa'>0$ and a strictly increasing map $\theta:\N\to\N$ such that $V_{\theta(n)}^+$ and $V_{\theta(n)+1}^+$ are at Euclidean distance at least $\kappa'$ for all $n$. But since both $\geom{\hex_n}{V_n}$ and $\geom{\hex_n}{V_{n+1}}$ diverge to $+\infty$ and we are in a configuration where $H_{V_n}$ is a diameter of $B^d$, we must have that $\hyp{L_{\theta(n)}}{V_{\theta(n)}}$ and $\hyp{U_{\theta(n)}}{V_{\theta(n)+1}}$ respectively get arbitrarily close to $V_{\theta(n)}^+$ and $V_{\theta(n)+1}^+$ for the Euclidean Hausdorff distance. In particular
\begin{equation*}
   \geom{\hex_{\theta(n)}}{U_{\theta(n)}} = d(H_{L_{\theta(n)}},H_{U_{\theta(n)}})\geq d\!\left(\hyp{L_{\theta(n)}}{V_{\theta(n)}},\hyp{U_{\theta(n)}}{V_{\theta(n)+1}}\right) \to +\infty
\end{equation*}
However, by the half-length property, we have $\geom{\hex_n}{U_n}\leq \frac23\trans{U}-\epsilon$ for all $n$, which is a contradiction.

\medskip

In view of Corollary \ref{cor:heredityNarrow}, we can assume that for all $n\geq N$, $\hex_n$ is not narrow, meaning that $\angle_{\side{L_n}}(\Ax U,\Ax V_n) \geq \pi/2$. Denote by $u_n = d(\Ax U,\Ax V_n)$, $a_n = \geom{\hor_n}{U}$, and $a=\frac13\trans{U}$ (by the half-length property, $a_n\geq a +\epsilon$). Note that $u_n\to 0$ and $\angle_{\side{L_n}}(\Ax U,\Ax V_n)\to\pi$.

Indeed, for any fixed $k$, in a configuration where $\side{L_n}$ and $\Ax U$ are diameters of $B^d$, the half-length property gives that the lines $\side{L_n},\ldots,\side{L_{n+k}}$ are aligned in this order along $\Ax U$, and that for all $i\in\{0,\ldots,k-1\}$, $d(\side{L_{n+i}},\side{L_{n+i+1}}) = a_{n+i}\geq a +\epsilon$. Therefore, the distance from $\side{L_n}$ to $\side{L_{n+k}}$ is at least $k(a + \epsilon)$.

   Now, consider $\P_{n,k}$ the hyperplane orthogonal to $\Ax U$ at distance $ka$ from $\side{L_n}$, and $2K_k$ the Euclidean diameter of its orthogonal projection onto $\side{L_n}$ in a configuration where $\side{L_n}$ and $\Ax U$ are diameters of $B^d$ (see Figure \ref{fig:LnRn}). Note that $K_k\to_{k\to\infty} 0$ as $ka\to\infty$. The last paragraph shows that $\P_{n,k}$ separates $\side{L_n}$ from $\side{L_{n+k}}$.

   Observe that because $d(\Ax V_n,\Ax V_{n+1})\to 0$ and $\trans{V_n} = n\trans{U} + t + o(1)$, if we are in a configuration where $\Ax U$ and $\side{L_n}$ are diameters of the ball, then the two points $\side{V_n}\cap\Ax V_n$ and $\side{V_n}\cap\Ax V_{n+1}$ are arbitrarily close to each other and to the sphere at infinity (for the Euclidean distance). Applying this $k$ times yields that, up to taking $n$ large enough, $V_n^-$ and $V_{n+k}^-$ are as close to each other for the Euclidean distance in $B^d$ as needed.

   But since $\hex_{n+k}$ is not narrow, $V_{n+k}^-\not\in \ohem{L_{n+k}}{U_{n+k}}$. Hence $V_{n+k}^-\in\H-\ohem{L_{n+k}}{U_{n+k}}$, and since $d(\P_{n,k},\hyp{L_{n+k}}{U_{n+k}}) = k\epsilon >0$, $V_n^-$ lies above $\P_{n,k}$ up to choosing $n$ large enough. Therefore, $\Ax V_n\cap \side{L_n}$ is the orthogonal projection of a point of $\P_{n,k}$ onto $\side{L_n}$, so $|u_n|\leq K_k$ if $n$ is large enough. Since this holds for any $k$ if $n$ is large enough, $u_n\to 0$. It also follows that $\angle_{\side{L_n}}(\Ax U,\Ax V_n)$ goes to $\pi$ since the projective model is conformal at the origin.

   \medskip

   By Lemma \ref{lm:additivityForRHHfarFromNarrow} applied to the right-angled hexagon $\hex_n$ with $S = U$, we have
\begin{equation*}
\geom{\hex_n}{V_{n+1}} + \geom{\hex_n}{U} = \geom{\hex_n}{V_n} + o(1)
\end{equation*}
so that, by the half-length property
\begin{equation*}
\frac12\trans{V_{n+1}} - \frac12\trans{V_n} \leq -\frac12\trans{U} + 3\left(\frac16\trans{U} - \epsilon\right) + o(1)
\end{equation*}
Finally, taking the limit on both sides yields
\begin{equation*}
\frac12\trans{U} \leq -3\epsilon
\end{equation*}
which is a contradiction.
\end{proof}

\begin{remark}
   In the case where $\rho$ is Coxeter extensible, we can cut short the proof by only using the non-generic case of Lemma \ref{lm:transLength}. We use the notation introduced in paragraph \ref{par:CoxExtReps}.

   Observe that $F_{V_n}=F_V$ for all $n$ and that $F_{U_n}$ is at distance $\frac {n+1} 2\trans{U}$ from $F_L$ towards $U^+$ along $\Ax U$.
   Indeed, if we have involutions $I_U$, $I_V$ and $I_L$ such that $U = I_UI_L$ and $V = I_LI_V$, and we write $I_{V_n} = I_V$, $I_{U_n} = I_U(I_LI_U)^n$ and $I_{L_n} = I_U(I_LI_U)^{n-1}$ for all $n\geq 1$, we easily check that they are all involutions (as conjugates of involutions), and that $V_n = U^nV = (I_UI_L)^nI_LI_V = I_U(I_LI_U)^{n-1}I_V = I_{L_n}I_{V_n}$ and $I_{U_n}I_{L_n} = I_U(I_LI_U)^nI_U(I_LI_U)^{n-1} = I_U(I_LI_U)^1I_U(I_LI_U)^0 = I_UI_L = U = U_n$ after observing that for all $k\geq 2$, $(I_LI_U)^kI_U(I_LI_U)^{k-1} = (I_LI_U)^{k-1}I_U(I_LI_U)^{k-2}$ and using induction.
   Thus $I_{V_n} = I_V$ and $I_{L_{n+1}} = I_{U_n}$ for all $n\geq 0$, and hence $F_{V_n} = F_V$ and $F_{U_n}$ is at positive distance $\frac {n+1} 2\trans{U}$ from $F_L$ along the axis of $U$ by induction.
   
   Now, if we suppose that $U^+\in\overline{F_V}$, then $I_VU^+ = U^+$, so $VU^+ = I_LI_VU^+ = I_LU^+ = U^-$, and hence $\trans{U^nV}\to 0$ by Lemma \ref{lm:transLength}, contradicting the $Q$-conditions. Therefore, $U^+\not\in\overline{F_V}$, so if $n$ is large enough, $F_{V_n} = F_V\subset\ohem{L_n}{U_n}$, hence $V_n^-\in\ohem{L_n}{U_n}$ \ie $L^n(U,V)$ is narrow.
\end{remark}

We can now prove the analogue of the trigonometric lemma used by Lee and Xu.

\begin{figure}[ht]
      \labellist
      \pinlabel $U_n$ [r] at 113 155
      \pinlabel $V_n$ [l] at 147 144
      \pinlabel $\side{L_n}$ [t] at 59 120
      \pinlabel $\hyp{U_n}{U_n}$ [t] at 70 207
      \pinlabel $\hyp{V_n}{V_n}$ [tl] <-1.5pt,1.5pt> at 164 173
      \endlabellist
      \centering
      \includegraphics[scale=.70]{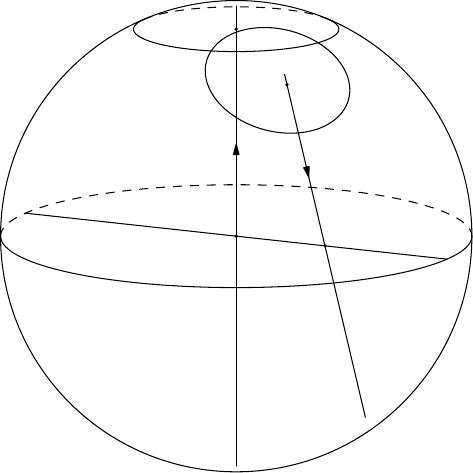}
      \caption{Since $\geom{\hex_n}{U_n}$ and $\geom{\hex_n}{V_n}$ are large, the Euclidean diameters of $\hyp{U_n}{U_n}$ and $\hyp{V_n}{V_n}$ are small. Thus, $\Ax V_n$ is close to $\Ax U_n$ for the Euclidean Hausdorff metric in this configuration}
      \label{fig:finitenessNonAcute1}
\end{figure}

\begin{proposition}\label{prop:finitenessNonAcute}
   If $\Lc_d([\rho])$ is large enough, all but finitely many Farey edges are narrow.
\end{proposition}
\begin{proof}
   Suppose there are infinitely many Farey edges $(U,V)$ such that $\hem{U}{U}\cap\hem{V}{V}\neq\varnothing$ ($\ast$). Since there are only finitely many Farey edges whose level is smaller than $N$ (and with Corollary \ref{cor:heredityNarrow} and Proposition \ref{prop:heredityThreshold}), we can find a Farey edge $(U,V)$ of level $N$ that satisfies ($\ast$), and that has infinitely many successors satisfying ($\ast$). We now define a sequence $(U_n,V_n)$ by setting $(U_0,V_0) = (U,V)$, and $(U_{n+1},V_{n+1}) = L(U_n,V_n)$ if $V^+(L(U_n,V_n))$ contains infinitely many Farey edges satisfying ($\ast$), and $(U_{n+1},V_{n+1}) = R(U_n,V_n)$ else.
   It is straightforward to check that the sequence $(U_n,V_n)$ is well-defined for $n\in\N$. We will denote by $\hex_n = \hex(U_n,V_n)$ the associated sequence of right-angled hexagons.

   We then make the following observations:
   \begin{enumerate}[(i), ref = \roman*]
      \item For all $n\in\N$, $\hem{U_n}{U_n}\cap\hem{V_n}{V_n}\neq\varnothing$ by Proposition \ref{prop:heredityThreshold} and Corollary \ref{cor:heredityNarrow}.
         \label{item:seqSatsisfiesStar}
      \item Both $\Lv U_n\to +\infty$ and $\Lv V_n\to +\infty$, otherwise, either $(U_n,V_n) = L^{n-m}(U_m,V_m)$ for some $m$ high enough and all $n\geq m$, or $(U_n,V_n) = R^{n-m}(u_m,v_m)$ instead. In both cases, Propositions \ref{prop:heredityThreshold} and \ref{prop:LnRn} show that this contradicts point \eqref{item:seqSatsisfiesStar} above.
         \label{item:LvGoesToInfty}
   \end{enumerate}

   Notice that by \eqref{item:LvGoesToInfty}, both $\Lv U_n$ and $\Lv V_n$ will be larger than $N$ if $n$ is large enough. Hence, up to supposing that $\Lc_d([\rho])$ is high enough, we can assume by the half-length property and Remark \ref{rmk:defLvN} that $\geom{\hex_n}{U_n}$ and $\geom{\hex_n}{V_n}$ are as large as needed.

   In particular, in a configuration where $\side{L_n}$ and $\Ax A_n$ are diameters of $B^d$, \eqref{item:seqSatsisfiesStar} and the last paragraph imply that up to supposing that $\Lc_d([\rho])$ is large enough, $d(\Ax U_n,\Ax V_n)$ and $\angle_{\side{L_n}}(\Ax U_n,\Ax V_n)$ are respectively as close to $0$ and $\pi$ as needed (see Figure~\ref{fig:finitenessNonAcute1}).

   We choose $\Lc_d([\rho])$ large enough that when $\Lv U_n$ and $\Lv V_n$ are larger than $N$, we have
   \begin{align*}
      d(\Ax U_n,\Ax V_n) &\leq K(2\epsilon) &\angle_{\side{L_n}}(\Ax U_n,\Ax V_n) &\geq \pi - K(2\epsilon)
   \end{align*}
   where we choose $\epsilon = E([\rho])$ optimal. In particular, by Lemma \ref{lm:additivityForRHHfarFromNarrow} applied to $\hex_n$ with $S = U_n$ if $(U_{n+1},V_{n+1}) = L(U_n,V_n)$, and $S = V_n$ otherwise, for $n$ large enough
   \begin{align*}
      \abs{\geom{\hex_n}{U_nV_n} + \geom{\hex_n}{U_n} - \geom{\hex_n}{V_n}} &\leq 2\epsilon &&\text{if $(U_{n+1},V_{n+1}) = L(U_n,V_n)$} \\
      \abs{\geom{\hex_n}{U_nV_n} + \geom{\hex_n}{V_n} - \geom{\hex_n}{U_n}} &\leq 2\epsilon &&\text{if $(U_{n+1},V_{n+1}) = R(U_n,V_n)$}
   \end{align*}
   By the half-length property, this implies that for $n$ large enough, if $(U_{n+1},V_{n+1}) = L(U_n,V_n)$, then
   \begin{align*}
      \abs{\trans{U_nV_n} + \trans{U_n} - \trans{V_n}} &\begin{multlined}[t]\leq \abs{\trans{U_nV_n}-2\geom{\hex_n}{U_nV_n})} + 4\epsilon\\
         \shoveright{+ \abs{\trans{U_n}-2\geom{\hex_n}{U_n}}}\\
         +\abs{\trans{V_n}-2\geom{\hex_n}{V_n}}
      \end{multlined}\\
      &\leq 3\left(\frac13 \transhex{\hex_n} - 2\epsilon\right) + 4\epsilon\\
      &\leq \transhex{\hex_n} - 2\epsilon
   \end{align*}
   and similarly, if $(U_{n+1},V_{n+1}) = R(U_n,V_n)$, we have when $n$ is large enough
   \begin{equation*}
      \abs{\trans{U_nV_n} + \trans{V_n} - \trans{U_n}} \leq \transhex{\hex_n} - 2\epsilon
   \end{equation*}
   This leads to a contradiction because it proves that the sequences $(\trans{U_n})_n$ and $(\trans{V_n}))_n$ are decreasing and both diverge to $-\infty$. Indeed, we see that if $(U_{n+1},V_{n+1}) = L(U_n,V_n)=(U_n,U_nV_n)$, then $\trans{V_{n+1}}\leq\trans{V_n}-\trans{U_n} + \transhex{\hex_n}-2\epsilon \leq \trans{V_n} - 2\epsilon$, while $\trans{U_{n+1}} = \trans{U_n}$ remains constant. Similarly, if $(U_{n+1},V_{n+1})=R(U_n,V_n)$, then $\trans{U_{n+1}}\leq\trans{U_n}-2\epsilon$ and $\trans{V_{n+1}}=\trans{V_n}$. Hence, point \eqref{item:LvGoesToInfty} above forces both sequences to decrease infinitely often by at least $2\epsilon$, so they must diverge to $-\infty$.
\end{proof}

From now on, $\Lc_d([\rho])$ is supposed to be large enough that the last proposition holds. This defines the function $\Lc_d$ that we need in Theorem \ref{thm:BQ+HLinPS}, which depends only on $\PrimSys([\rho])$ and $E([\rho])$. The fact that $\Lc_d$ can be chosen to be  a non-increasing function of both $\PrimSys([\rho])$ and $E([\rho])$ comes from the fact that $\mu\mapsto R(\mu)$ and $\mu\mapsto K(\mu)$ are respectively decreasing and increasing. Since $E([\rho])\leq \frac16\PrimSys([\rho])$ by Equation \eqref{eq:HLP}, if $\PrimSys([\rho])\to 0$, then $E([\rho])\to 0$ as well, so $\Lc_d([\rho])\to + \infty$ as $R(\mu)\to +\infty$ and $K(\mu)\to 0$ when $\mu\to 0$.

\subsection{Hyperplanes ordering}\label{par:propO}

A key observation in Lee and Xu's proof is that it is possible to order the translates of a hyperplane by the elements of $\Ax_{(u,v)} w$ where $(u,v)$ is a high level Farey edge and $w\in V^+(u,v)$ is a positive word in $(u,v)$. This relies on the following property.

\begin{definition}
   A Farey edge $(U,V)$ is said to have \emph{property $\mathcal{O}$} if there exists a hyperplane $\mathcal{H}$ such that for all $S,T\in\{U,V\}$, the three hyperplanes $S^{-1}\mathcal{H}$, $\mathcal{H}$ and $T\mathcal{H}$ are pairwise disjoint and $\mathcal{H}$ separates the other two.
\end{definition}

We will simply write $S^{-1}\mathcal{H},\mathcal{H},T\mathcal{H}$ to symbolize the above relations between those three hyperplanes.

We now show the connection between narrow bases and bases having property $\mathcal{O}$.

\begin{proposition}\label{prop:narrowImplyPropO}
   If $(U,V)$ is a narrow basis of level larger than $N$, then both $L(U,V)$ and $R(U,V)$ satisfy property $\mathcal{O}$.
\end{proposition}
\begin{proof}
Write $(U',V') = L(U,V)$, $\hex=\hex(U,V)$ and $\hex'=\hex(U',V')$. We define $H'$ to be the hyperplane orthogonal to $\Ax U$ in the middle of the segment delimited by $\hyp{L}{U}$ and $\hyp{U}{U}$. We start by using repeatedly Lemma \ref{lm:hypNotInt} to show
   \begin{description}[style=nextline, leftmargin=.7cm]
      \item[$\bm{H'\cap \hyp{V}{V} = \varnothing}$]
         If $\Lv U = \Lv (U,V)$, then $(\Delta,\P,\P',\Q) = (\side{L},\hyp{L}{V},\hyp{V}{V},H')$ yields the result, as $d(H_L,\hyp{V}{V})=d(\hyp{L}{V},\hyp{V}{V}) = \geom{\hex}{V}\geq \frac13\trans{V}+\epsilon\geq \frac16\PrimSys([\rho])$, and $d(H_L,H') = \frac12 d(H_L,\hyp{U}{U}) =\frac12\geom{\hex}{U}\geq\frac 16\trans{U} + \frac12\epsilon\geq R(\frac16\PrimSys([\rho]))$, where the last inequality comes from Remark \ref{rmk:defLvN} and the fact that $\Lv U\geq N$. If $\Lv V = \Lv (U,V)$, $(\Delta,\P,\P',\Q) = (\side{L},\hyp{L}{U},H',\hyp{V}{V})$ is the right choice, as $d(H_L,H') = d(\hyp{L}{U},H')\allowbreak = \frac12\geom{\hex}{U}\geq\frac16\trans{U} + \epsilon\geq\frac16\PrimSys([\rho])$, and $d(H_L,\hyp{V}{V}) = \geom{\hex}{V}\geq\frac13\trans{V} + \epsilon\geq R(\frac16\PrimSys([\rho]))$ by Remark \ref{rmk:defLvN} since $\Lv V\geq N$. In both cases, the narrowness of $(U,V)$ allows to use Lemma \ref{lm:hypNotInt}.\eqref{item:lm2} to conclude.

         Observe that this implies that $\Ax V'\cap H'\neq\varnothing$ as $\Ax V' = \Perp(H_V,H_U)$ and $H'$ separates $H_U\subset\hyp{U}{U}$ and $H_V\subset\hyp{V}{V}$ since $(U,V)$ is narrow.
\item[$\bm{\hyp{L'}{U'}\cap\hyp{V'}{V'}=\varnothing}$]
   We check similarly that we can apply Lemma \ref{lm:hypNotInt}.\eqref{item:lm1} with $(\Delta,\P,\P',\Q) = (\side{L'},\hyp{L'}{U'},H',\hyp{V'}{V'})$ to get the expected result, since $\Lv V' = \Lv(U',V')\geq N$ and $\Ax V'\cap H'\neq\varnothing$.
\item[$\bm{\hyp{L'}{U'}\cap V'\hyp{V'}{V'}=\varnothing}$]
   Again, Lemma \ref{lm:hypNotInt}.\eqref{item:lm1} applied to $(\Delta,\P,\P',\Q) = (\side{L'},\hyp{L'}{U'},H',V'\hyp{V'}{V'})$ gives the result. Indeed, we have as before that $d(H_{L'},H') = d(\hyp{L'}{U'},H') = d(\hyp{U}{U},H')\allowbreak = \frac12\geom{\hex}{U}\geq\frac16\trans{U} + \epsilon\geq \frac16\PrimSys([\rho])$, and we also have $d(H_{L'},V'\hyp{V'}{V'}) = \trans{V'} - d(H_{L'},\hyp{V'}{V'}) = \trans{V'} - \geom{\hex'}{V'} \geq \frac13\trans{V'} + \epsilon\geq R(\frac16\PrimSys([\rho]))$ by Remark \ref{rmk:defLvN} since $\Lv V'\geq N$. Moreover, $\Ax V'\cap H'\neq\varnothing$ allows to use case \eqref{item:lm1} since $\Ax V' = \Perp(H_{L'},\hyp{V'}{V'}) = \Perp(H_{L'},V'\hyp{V'}{V'})$.
\end{description}

The last two results give $V'^{-1}\hyp{L'}{U'}\cap\hyp{L'}{U'} = \varnothing$, so that $U'\hyp{L'}{U'},\hyp{L'}{U'},V'^{-1}\hyp{L'}{U'}$.

Since $V'^{-1}\hyp{L'}{U'}\cap\hyp{L'}{U'} = \varnothing$, we also have $\hyp{L'}{U'}\cap V'\hyp{L'}{U'} = \varnothing$, so that we obtain $V'\hyp{L'}{U'},\hyp{L'}{U'},\allowbreak U'^{-1}\hyp{L'}{U'}$ as well. We trivially have that $U'\hyp{L'}{U'},\hyp{L'}{U'},U'^{-1}\hyp{L'}{U'}$, and $V'^{-1}\hyp{L'}{U'},\hyp{L'}{U'},\allowbreak V'\hyp{L'}{U'}$ follows from $V'^{-1}\hyp{L'}{U'}\cap\hyp{L'}{U'} = \varnothing$ and $\hyp{L'}{U'}\cap V'\hyp{L'}{U'} = \varnothing$. Hence, $L(U,V)$ has property $\mathcal{O}$.

We prove in a similar way that $R(U,V)$ satisfies $\mathcal{O}$.
\end{proof}
We get as a corollary (using Proposition \ref{prop:finitenessNonAcute})
\begin{corollary}\label{cor:finitenessNonO}
All but finitely many Farey edges satisfy property $\mathcal{O}$.
\end{corollary}
From now on, the proof is the same as the one given by Lee and Xu. We include it for the sake of completeness.

\subsection{End of the proof (from Lee--Xu)}
\begin{fact}[{see \cite[\Thm 5.4]{LeeXu2020}}]\label{fact:highLv} If $f=(x,y)\in E$ has property $\mathcal{O}$, then there exist positive constants $m_f$ and $c_f$ and a point $o_f\in \H$ such that for every positive word $w$ in $\{x,y\}$ the inequality
\begin{equation*}
m_f d_e(u,v)-c_f\leq d(\rho(u)o_f,\rho(v)o_f)
\end{equation*}
holds for every $u,v\in\Ax w$.
\end{fact}
\begin{proof}
Choose a point $o_f\in\mathcal{H}$, and a word $w$ positive in $\{x,y\}$. $w$ is positive in $f$, so that if $g,h\in\Ax_f w$ are consecutive (in this order), then $g^{-1}h\in\{x,y\}$. In fact, if $w = w_1\cdots w_n$ where $w_i\in\{x,y\}$ for all $i$, then 
\begin{equation*}
\Ax_f w = \{\ldots,w^{-1}w_n^{-1},w^{-1},\ldots,w_n^{-1}w_{n-1}^{-1},w_n^{-1},1,w_1,w_1w_2,\ldots,w,ww_1,\ldots\}
\end{equation*}

Hence, for all $v\in\Ax_f(w)$, we find $g,h\in\{x,y\}$ such that $(vg^{-1},v,vh)\subset \Ax_f w$. Property $\mathcal{O}$ then tells us that $\rho(vg^{-1})\mathcal{H},\rho(v)\mathcal{H},\rho(vh)\mathcal{H}$, which implies that $\rho(\Ax_f w)\mathcal{H}$ is a well ordered set of pairwise disjoint hyperplanes.

Thus, if we write $m = \min_{u\in\{x,y\}}d(\mathcal{H},\rho(u)\mathcal{H})$, we see that
\begin{equation*}
md_f(u,v)\leq d(\rho(u)o_f,\rho(v)o_f)
\end{equation*}
for all $u,v\in\Ax_f w$.
The triangle inequality then allows to change $d_f$ into $d_e$ and $\Ax_f w$ into $\Ax w$.
\end{proof}

The proof of Theorem \ref{thm:BQ+HLinPS} now follows easily.
\begin{proof}[Proof of Theorem \ref{thm:BQ+HLinPS}]
Take $\rho: F_2\to \Isom(\H)$ a representation that satisfies both the $Q$-conditions and the half-length property, and choose a point $o\in\H$. We know from Corollary \ref{cor:finitenessNonO} that there exists $N$ such that if $[x,y]\in E$ is a Farey edge with $\Lv[x,y]\geq N$, then $[x,y]$ satisfies $\mathcal{O}$.

Denote by $\{f_i\}_{1\leq i\leq 2^{N+1}}$ the set of level $N$ Farey edges. Then we can write
\begin{equation*}
V = V_{< N}\cup V^+[f_1]\cup\cdots\cup V^+[f_{2^{N+1}}]
\end{equation*}
where $V_{<N}$ is the (finite) set of Christoffel words of level smaller than $N$. For all integers $i\in[1;2^{n+1}]$, let $o_i\in \H$ and $m_i,c_i>0$ be the point and constants given by Fact \ref{fact:highLv} for the Farey edge $f_i$, and write $r_i = d(o,o_i)$. We see with the triangle inequality that
\begin{align*}d(\rho(u)o,\rho(v)o) &\geq d(\rho(u)o_i,\rho(v)o_i)-d(\rho(u)o,\rho(u)o_i)-d(\rho(v)o,\rho(v)o_i)\\
&\geq m_id_e(u,v)-c_i-2r_i
\end{align*}
for all $u,v\in\Ax w$ where $w\in V^+[f_i]$.

Now, if $w\in V_{<N}$, since $\rho(w)$ is hyperbolic, we can find positive constants $m_w$ and $c_w$ such that
\begin{equation*}
m_wd_e(u,v)-c_w\leq d(\rho(u)o,\rho(v)o)
\end{equation*}
for all $u,v\in\Ax w$.

Thus, if we write 
\begin{align*}
m &=\min\{m_w,w\in V_{<N}\}\cup\{m_i,1\leq i\leq 2^{N+1}\}\\
c &=\max\{c_w,w\in V_{<N}\}\cup\{c_i+2r_i,1\leq i\leq 2^{N+1}\}
\end{align*}
we get that for every $w\in V$,
\begin{equation*}
md_e(u,v)-c\leq d(\rho(u)o,\rho(v)o)
\end{equation*}
for all $u,v\in\Ax w$.
\end{proof}

Theorems \ref{thm:PSisPDisBQ} and \ref{thm:PSisBQCox} follow.
\begin{proof}[Proof of Theorems \ref{thm:PSisPDisBQ} and \ref{thm:PSisBQCox}]
   Combining Propositions \ref{prop:PSinPDinBQ} and \ref{prop:BQisIrr} with Theorem \ref{thm:BQ+HLinPS} proves theorem \ref{thm:PSisPDisBQ}.

   Theorem \ref{thm:PSisBQCox} is an easy corollary of Theorem \ref{thm:PSisPDisBQ}.
\end{proof}

\section{The bounded intersection property}\label{sec:BIP}

We can also prove Theorem \ref{thm:PSinBIP} using the same proof as that of \cite[\Thm II]{LeeXu2020}. We do not give here all of the definitions and details concerning the palindromic primitive words and the bounded intersection property, and refer to Section 6 of \cite{LeeXu2020} for full details.

Lee and Xu start by modifying the choice of primitive words representatives as in \cite{GilmanKeen2011} so that all representatives with odd word length are palindromic in the basis $(a,b)$. This amounts to the following modifications in the definitions of the transformations $L$ and $R$:

\begin{definition}
   If $(u,v)$ is a Farey edge, then
   \begin{align*}
      L(u,v) &= \begin{cases}
         (u,vu) &\text{if $|uv|_e$ is even}\\
         (u,uv) &\text{otherwise}
      \end{cases} &
      R(u,v) &= \begin{cases}
         (vu,v) &\text{if $|uv|_e$ is even}\\
         (uv,v) &\text{otherwise}
      \end{cases}
   \end{align*}
\end{definition}

Since $c = (ab)^{-1}$ has an even size in the basis $e$, it follows that applying this algorithm for each of the bases $(a,b)$, $(b,c)$ and $(c,a)$ gives 3 different representatives to each element of $V$, two of which are palindromes (see \cite[\Lm 6.6]{LeeXu2020} and the following sentence). In order to ease the redaction, we will say that a primitive element $w$ is $L$-palindromic if it is palindromic in $(a,b)$, $A$-palindromic if it is palindromic in $(c,a)$, and $B$-palindromic if it is palindromic in $(b,c)$.

Building on this observation and \cite{GilmanKeen2009}, both Lee and Xu and Series define the sets $J_L$, $J_A$ and $J_B$ as follows in the case where $\rho: F_2\to\PSL_2(\C)$, relying on \cite[\Lm 6.8]{LeeXu2020} and \cite[\Lm 5.1]{GilmanKeen2009} (see also \cite[\Rmk 6.9]{LeeXu2020}).
\begin{align*}
   J_L &= \{H_L\cap\Ax\rho(w),\text{ for $w\in V$ $L$-palindromic}\}\\
   J_A &= \{H_A\cap\Ax\rho(w),\text{ for $w\in V$ $A$-palindromic}\}\\
   J_B &= \{H_B\cap\Ax\rho(w),\text{ for $w\in V$ $B$-palindromic}\}
\end{align*}

An irreducible representation $\rho: F_2\to\PSL_2(\C)$ has the bounded intersection property (for Lee and Xu) if each of the associated sets $J_A$, $J_B$ and $J_L$ are bounded and all primitive elements are sent to hyperbolic isometries.

However, \cite[\Lm 6.8]{LeeXu2020} does not hold \emph{a priori} in our context for representations that are not Coxeter extensible, so $J_L$, $J_A$ and $J_B$ are no longer well-defined. Thus, we need to reinterpret them in order to get a new definition that makes sense in general.

We start by observing that when $\rho$ is Coxeter extensible and the fixed point sets $F_A$, $F_B$ and $F_L$ are of dimension $1$ (we can always suppose it is the case when $\rho:F_2\to \PSL_2(\C)$, by \cite[\pp 47]{Fenchel1989}), they correspond exactly to $\side{A}$, $\side{L}$ and $\side{B}$.

In this situation, the set $J_L$ is bounded if and only if there exists $M$ such that for all primitive element $w$ palindromic in $(a,b)$, $d(o_L,\Ax\rho(w))\leq M$ where $o_L = \side{L}\cap \Ax A$. Indeed, $a$ is primitive and palindromic in $(a,b)$, hence $o_L\in J_L$, and we know from \cite[\Lm 6.8]{LeeXu2020} that all the intersections are orthogonal along $\side{L} = \Ax I_L$ (hence the distance from $o_L$ to $\Ax\rho(w)$ for $w$ as above is realized along $\side{L}$).

The same argument shows that $J_A$ (\resp $J_B$) is bounded if and only if there exists $M$ such that for all primitive element $w$ palindromic in $(c,a)$ (\resp $(b,c)$), $d(o_A,\Ax\rho(w))\leq M$ (\resp $d(o_B,\Ax\rho(w))\leq M$), where $o_A =\side{A}\cap\Ax AB$ and $o_B = \side{B}\cap\Ax B$.

This motivates our new definition:
\begin{definition}\label{def:BIP}
   A representation $\rho: F_2\to\Isom(\H)$ satisfies the \emph{bounded intersection property} if
\begin{enumerate}
\item $\rho$ is irreducible;
\item the images of primitive elements are hyperbolic;
\item there exists $M$ such that for all $X\in\{A,L,B\}$, and for all $X$-palindromic primitive element $w$, $d(o_X,\Ax\rho(w))\leq M$.
\end{enumerate}
\end{definition}

\begin{remark}
   Observe that contrary to what is claimed by Series in \cite[\Rmk 1.6]{Series2020}, neither Lee and Xu's result, nor our is flawed. The misunderstanding comes from the fact that Lee and Xu's definition is slightly stronger than Series', who doesn't require that a representation satisfying the bounded intersection property sends primitive elements to hyperbolic elements. As in the rest of this paper, we chose to follow Lee and Xu's definition in this section.
\end{remark}

It is not hard to show, following the proof of \cite[\Prop 6.11]{LeeXu2020}, that this defines an $\Out(F_2)$-invariant subset of the character variety $\chi(F_2,\Isom(\H))$.

We are now able to prove Theorem \ref{thm:PSinBIP}. The proof is the same as that given by Lee and Xu.

\begin{proof}[Proof of Theorem \ref{thm:PSinBIP}]
   Suppose $\rho$ is primitive stable. This means that for some point $o\in\H$, the orbit map $\tau_o:u\in F_2\mapsto \rho(u)o$ is uniformly a quasi-isometric embedding in restriction to the axes of primitive elements in $F_2$. Hence, the Morse lemma gives the existence of some $D$ such that for all primitive element $w$, and for all $u\in\Ax_e w$, $d(\rho(u)o,\Ax\rho(w))\leq D$.

Now, suppose without loss of generality that $w$ is $L$-palindromic. Since it is palindromic, it is cyclically reduced, so $1\in\Ax_e w$ by Proposition \ref{prop:AxisInF2}. Hence, $d(o,\Ax\rho(w))\leq D$, and $d(o_L,\Ax\rho(w))\leq D + d(o,o_L)$. Thus, $\rho$ has the bounded intersection property with $M = D + \max\{d(o,o_L),d(o,o_A),d(o,o_B)\}$ ($\rho$ satisfies items (i) and (ii) by Propositions \ref{prop:PSinPDinBQ} and \ref{prop:BQisIrr}).

\medskip

Now suppose $\rho$ has the bounded intersection property and is discrete. The images of primitive elements are hyperbolic by Definition \ref{def:BIP}. Choose any $\lambda >0$, and suppose $\{w\in V, \trans{\rho(w)}\leq \lambda\}$ is infinite. Without loss of generality, we may suppose that
\begin{equation*}
W = \{\text{$w\in V$ is $L$-palindromic}, \trans{\rho(w)}\leq \lambda\}
\end{equation*}
is infinite as well.

But for all $w\in W$, there exists $n\in\Ax\rho(w)$ such that $d(o_L,n)\leq M$ by the bounded intersection property. Hence,
\begin{align*}
d(o_L,\rho(w)o_L) &\leq d(o_L,n) + d(n,\rho(w)n) + d(\rho(w)n,\rho(w)o_L)\\
&\leq 2M + \lambda
\end{align*}
for every $w\in W$. This contradicts the discreteness of $\rho$, so $\rho$ satisfies the $Q_\lambda$-conditions for every $\lambda > 0$.
\end{proof}

\section{Further questions}\label{sec:furtherQuestions}

Little is known about the half-length property apart from the fact that Coxeter extensible representations satisfy it. It is to be thought of as some sort of weakening of being Coxeter extensible, and the hope is that it defines at least a neighborhood of the Coxeter extensible locus. Computer experiments seem to suggest that a lot of Schottky representations satisfy it, but it remains unclear how to prove this.

One could also consider generalizations of this proof in different settings. In this direction, one can observe that besides the half-length property, the proof relies on two sets of assumptions. The first one is that $X$ needs to be Gromov hyperbolic for the most critical parts of the argument (see Propositions \ref{prop:LnRn} and \ref{prop:finitenessNonAcute}). It is not hard to convince oneself that if the space is hyperbolic, much of the geometric definitions and properties involving orthogonality are easily adapted by considering closest point projections.

Almost every part of the argument goes through in this setting, in fact, apart from Lemma \ref{lm:hypNotInt}, which fails in the case of trees, because the orthogonal hyperplanes get “too thick”. An easy way to make sure this lemma holds is to work with Gromov hyperbolic spaces for which a similar picture is valid. In particular, we want these orthogonal hyperplanes to be somewhat “thin”.

A plausible candidate would be Gromov hyperbolic Hilbert geometries. In fact, Benoist's characterization of Gromov hyperbolic Hilbert geometries as quasi\-sym\-metric convex sets (see \cite[\Thm 1.4]{Benoist2003}) gives good hope that one could in this case use much of the argument exposed here as is. Indeed, it seems that \cite[\Prop 1.8]{Benoist2003} implies that we can define a uniform notion of “rough polars” to projective hyperplanes, which we may use to approximate closest point projections of a point onto this hyperplane. It seems plausible that the same proof would work in this setting.

\printbibliography

\end{document}